\documentclass[11pt]{article}

\usepackage{hyperref}

\usepackage{amsfonts}
\usepackage{amssymb,amsmath}
\usepackage{ stmaryrd }
\usepackage[T1]{fontenc}

 \usepackage[utf8]{inputenc} 

\usepackage{amsthm}

\usepackage{bbm}

\usepackage{dsfont}

\usepackage{color}
\definecolor{gray5}{gray}{0.8}
\definecolor{gray1}{gray}{0.4}
\definecolor{gray2}{gray}{0.6}
\definecolor{gray3}{gray}{0.7}
\definecolor{gray4}{gray}{0.3}

\numberwithin{equation}{section}

\newcommand     {\NN}{\mathbb{N}}
\newcommand     {\RR}{\mathbb{R}}
\newcommand     {\PP}{\mathbb{P}}

\newcommand     {\EE}{\mathbb{E}}

\newtheorem     {thm}{Theorem}[section]

\newtheorem     {lem}[thm]{Lemma}
\newtheorem     {prop}[thm]{Proposition}
\newtheorem     {cor}[thm]{Corollary}

\newcommand{\norm}[1]{\|#1 \|}

\newcommand{\ee}{\varepsilon}

\begin{document}

\title{A probabilistic approach to Dirac concentration in nonlocal models of adaptation with several resources} 

\author{\textsc{Nicolas Champagnat$^{1,2,3}$, Beno\^it Henry$^{1,2}$}}

\footnotetext[1]{IECL, Universit\'e de Lorraine, Site de Nancy, B.P. 70239, F-54506 Vandœuvre-lès-Nancy Cedex, France\\ E-mail:
  \texttt{Nicolas.Champagnat@inria.fr}, \texttt{benoit.henry@univ-lorraine.fr}}
\footnotetext[2]{CNRS, IECL, UMR 7502, Vand{\oe}uvre-l\`es-Nancy, F-54506, France}  
\footnotetext[3]{Inria, TOSCA team, Villers-l\`es-Nancy, F-54600, France}

\date{}
\maketitle

\begin{abstract}
  This work is devoted to the study of scaling limits in small mutations and large time of the solutions $u^\ee$ of two deterministic
  models of phenotypic adaptation, where the parameter $\varepsilon>0$ scales the size of mutations. The first model is the so-called
  Lotka-Volterra parabolic PDE in $\RR^d$ with an arbitrary number of resources and the second one is an adaptation of the first
  model to a finite phenotype space. The solutions of such systems typically concentrate as Dirac masses in the limit $\ee\to 0$. Our
  main results are, in both cases, the representation of the limits of $\ee \log u^{\ee}$ as solutions of variational problems and
  regularity results for these limits. The method mainly relies on Feynman-Kac type representations of $u^\ee$ and Varadhan's Lemma.
  Our probabilistic approach applies to multi-resources situations not covered by standard analytical methods and makes the link
  between variational limit problems and Hamilton-Jacobi equations with irregular Hamiltonians that arise naturally from analytical
  methods. The finite case presents substantial difficulties since the rate function of the associated large deviation principle has
  non-compact level sets. In that case, we are also able to obtain uniqueness of the solution of the variational problem and of the
  associated differential problem which can be interpreted as a Hamilton-Jacobi equation in finite state space.
\end{abstract}          
\bigskip



\section{Introduction}

We are interested in the dynamics of a population subject to mutations and selection driven by competition for resources. Each
individual in the population is characterized by a quantitative phenotypic trait $x\in\RR$ (for example the size of individuals,
their mean size at division for bacteria, their rate of nutrients intake or their efficiency in nutrients assimilation). 

The partial differential equations we study are mutation-competition models taking the form of reaction-diffusion
equations, with non-local density-dependence in the growth rate and which have been studied in various contexts by a lot of authors.
The general equation takes the form
\begin{equation}
  \label{eq:EDP}
  \partial_{t}u^{\ee}(t,x)=\frac{\ee}{2} \Delta u^{\ee}(t,x)+\frac{1}{\ee}u^{\ee}(t,x)R\left(x,v^{\ee}_{t}\right),\quad \forall t>0,\ x\in\RR^d 
\end{equation}
where $x$ corresponds to the phenotypic traits characterizing individuals, $u^\ee(tx)$ is the density of population with trait $x$ at
time $t$ and $v^{\ee}_{t}=\left(v^{1,\ee}_{t},\dots,v^{r,\ee}_{t} \right)$ with
\[
v^{i,\ee}_{t}=\int_{\mathbb{R}^{d}}\Psi_{i}\left(x\right)u^{\ee}\left(t,x\right)dx, \quad 1\leq i\leq r,
\]
for some functions $\Psi_i:\RR^d\rightarrow\RR_+$. In~\eqref{eq:EDP}, the Laplace operator models mutation and $R(x,v^{\ee}_{t})$ is
the growth rate of individuals with trait $x$ at time $t$. Competition occurs through the functions $v^{i,\varepsilon}_t$, which
depend on $u^\varepsilon$. A typical example of function $R$ is given by
\begin{equation}
  \label{eq:chemostat}
  R(x,v)=\sum_{i=1}^r \frac{c_i\Psi_i(x)}{1+v_i}-d(x),
\end{equation}
where the first term models births which occur through the consumption of $r$ resources whose concentrations at time $t$ are given by
$c_i/(1+v_i)$ and with a trait-dependent consumption efficiency given by the function $\Psi_i(x)$, and the second term corresponds to
deaths without competition at trait-dependent rate $d(x)$. This form of the function $R$ is relevant for populations of
micro-organisms in a chemostat, and has been studied for related models in lots of
works~\cite{diekmann-jabin-al-05,champagnat-jabin-raoul-10,champagnat-jabin-11,champagnat-jabin-meleard-13}.

The parameter $\varepsilon>0$ in~\eqref{eq:EDP} introduces a scaling (in the limit $\varepsilon\rightarrow 0$) of small or rare
mutations (this is the same for the Laplace operator) and of large time, which was introduced in models of adaptive dynamics
in~\cite{diekmann-jabin-al-05}, and has been used since a long time to study front propagation in standard reaction-diffusion
problems~\cite{fleming-souganidis-86,barles-evans-al-90,freidlin-85,freidlin-92}. In our model, the qualitative outcome is that
solutions concentrate as Dirac masses, and this concentration is studied using the WKB ansatz
\[
u^\varepsilon(t,x)=\exp\left(\frac{\varphi^\varepsilon(t,x)}{\varepsilon}\right)
\]
in~\cite{diekmann-jabin-al-05,barles-perthame-07,perthame-barles-08,barles-mirrahimi-al-09,lorz-mirrahimi-al-11,champagnat-jabin-11,mirrahimi-perthame-al-12}
for different particular cases of~\eqref{eq:EDP} and also
in~\cite{perthame-genieys-07,desvillettes-jabin-al-08,jabin-raoul-11,raoul-11,lorz-mirrahimi-al-11} for models with competitive
Lotka-Volterra competition. Several of these works prove the convergence along a subsequence $(\varepsilon_k)_{k\geq 1}$
converging to 0 of $\varphi^\varepsilon$ to a solution $\varphi$ of the Hamitlon-Jacobi problem
\begin{equation}
  \label{eq:HJ-intro}
  \partial_t\varphi(t,x)=R(x,v_t)+\frac{1}{2}|\nabla\varphi(t,x)|^2,  
\end{equation}
where $v_t$ is expected to take the form
\[
v^i_t=\int_{\mathbb{R}^{d}}\Psi_{i}(x)\mu_t(dx),
\]
where $\mu_t$ is some (measure, weak) limit of $u^{\varepsilon_k}(t,x)$. Due to the fact that, under general assumptions, the total mass of
the population $\int_{\RR^d} u^\varepsilon(t,x)dx$ is uniformly bounded and bounded away from 0, the function $\varphi$ satisfies the
constraint $\sup_{x\in\RR^d}\varphi(t,x)=0$ for all $t\geq 0$, and the measure $\mu_t$ is expected to have support in
$\{\varphi(t,\cdot)=0\}$. In addition, the measure $\mu_t$ is expected to be metastable in the sense that $R(x,v_t)\leq 0$ for all
$x$ such that $\varphi(t,x)=0$ and $R(x,v_t)=0$ for all $x$ in the support of $\mu_t$ (to preserve the condition
$\sup_{x\in\RR^d}\varphi(t,x)=0$).

However, the study of the Hamilton-Jacobi problem~\eqref{eq:HJ-intro} with (some or all of) the previous constraints is a difficult
problem. For example, uniqueness is only known in general in the case
$r=1$~\cite{mirrahimi-roquejoffre-16,CRAS-mirrahimi-roquejoffre} (see also~\cite{perthame-barles-08}). In addition, all the previous
references only prove the convergence to~\eqref{eq:HJ-intro} for $r=1$, except in~\cite{champagnat-jabin-11} where the very specific
model~\eqref{eq:chemostat} is studied for any values of $r\geq 2$. Yet, the case $r\geq 2$ is of particular biological interest since
it is the only case where a phenomenon of diversification known as \emph{evolutionay branching}~\cite{metz-geritz-al-96,dieckmann-doebeli-99} can occur (see~\cite[Prop.\,3.1]{champagnat-jabin-meleard-13}). \medskip

In this article, after describing the general model and the standing assumptions and recalling its basic properties in
Section~\ref{sec:problem}, we propose a different approach to study the convergence of $\varepsilon\log u^\varepsilon(t,x)$ as
$\varepsilon\rightarrow 0$. This approach is based on a probabilistic interpretation of the solution of the PDE~\eqref{eq:EDP}
through a Feynman-Kac formula involving some functional of the stochastic process corresponding to mutations, here Brownian motion
(see Section~\ref{sec:FK}). This formula suggests to use large deviations, and more specifically Varadhan's lemma, to study the
convergence of $\varepsilon\log u^\varepsilon(t,x)$, giving in the limit a variational problem which takes the standard form of
variational problems associated to Hamilton-Jacobi equations, in Section~\ref{sec:Varadhan}. This general idea is actually not new
since it goes back to works of Freidlin~\cite{freidlin-85,freidlin-92} on reaction-diffusion equations. In~\cite{freidlin-87},
Freidlin also studied similar questions for models close to~\eqref{eq:EDP} (with different initial conditions), but only for a single
ressource ($r=1$) and under the assumption that, for all $x\in\RR^d$, there exists a unique $a(x)\in\RR^r$ such that $R(x,a(x))=0$.
This condition is also assumed in the more recent
works~\cite{perthame-barles-08,barles-mirrahimi-al-09,lorz-mirrahimi-al-11,mirrahimi-perthame-al-12}. These assumptions are not
needed in our study. We also present our results in a more unified framework, avoiding in particular the use of precise properties of
the heat semi-group. It is therefore easy to extend our results to other mutation operators in~\eqref{eq:EDP}, as discussed in
Section~\ref{sec:extensions}. In cases where the convergence to the Hamilton-Jacobi problem is known, we also deduce as a side result
the equality between the solution to the Hamilton-Jacobi problem and its variational formulation (see Section~\ref{sec:HJ-VAR}).
Interestingly, this result does not seem to be covered by existing general results on this topic because of the possible
discontinuities of the coefficients of the Hamilton-Jacobi problem.

A natural extension is the case of finite trait space $E$ (instead of $\RR^d$), in which the large deviations scaling suggests to
consider
\[
\dot{u}^{\ee}(t,i)=\sum_{j\in
  E}\exp\left(-\frac{\mathfrak{T}(i,j)}{\ee}\right)(u^{\ee}(t,j)-u^{\ee}(t,i))+\frac{1}{\ee}u^{\ee}(t,i)R(i,v^{\ee}_{t}),\quad
\forall t\in[0,T],\ \forall i\in E
\]
for positive $\mathfrak{T}(i,j)$. The extension needs some care since in this case the large deviations principle for the mutation
process is not standard (in particular, the rate function has non-compact level sets). This is done in Section~\ref{sec:discrete}.
Since the variational problem is simpler in this case, we can study it in detail. In Section~\ref{sec:uniqueness-E-finite}, we prove
that, under assumptions ensuring that one can associate a unique metastable measure $\mu_t$ to any set of zeroes of
$\varphi(t,\cdot)$, the variational problem and the associated Hamilton-Jacobi problem both have a unique solution with appropriate
regularity. Hence the full family $(\varepsilon\log u^\varepsilon(t,x))_{\varepsilon>0}$ converges to this limit. A key point
consists in proving that one can characterize any accumulation point of $v^\varepsilon_{t+s}$ for small $s>0$ only from the zeroes of
$\varphi(t,\cdot)$.

\section{Problem statement and preliminray results}
\label{sec:problem}
We consider the following partial differential equation in $\RR_+\times\RR^d$:
\begin{equation}
\label{eq:prob}
\left\{
\begin{array}{ll}
\partial_{t}u^{\ee}(t,x)=\frac{\ee}{2} \Delta u^{\ee}(t,x)+\frac{1}{\ee}u^{\ee}(t,x)R\left(x,v^{\ee}_{t}\right),\quad \forall t>0,\ x\in\RR^d &  \\ 
u^{\ee}(0,x)=\exp\left(-\frac{h_\varepsilon(x)}{\ee} \right), \quad\forall x\in\RR^d& 
\end{array} \right.
\end{equation}
with
\[
v^{\ee}_{t}=\left(v^{1,\ee}_{t},\dots,v^{r,\ee}_{t} \right),
\]
where
\[
v^{i,\ee}_{t}=\int_{\mathbb{R}^{d}}\Psi_{i}\left(x\right)u^{\ee}\left(t,x\right)dx, \quad 1\leq i\leq r,
\]
$R$ is a map from $\RR^d\times\RR^r$ to $\RR$ and $\Psi_i$ and $h_\varepsilon$ are maps from $\RR^d$ to $\RR$.

\medskip

Let us state our assumptions of $R$, $\psi_{i}$ and $h_\varepsilon$. 

\medskip
\begin{enumerate}
\item{\bf Assumptions on $\Psi_{i}$}

\medskip

$\bullet$ There exist $\Psi_{\text{min}}$ and $\Psi_{\text{max}}$, two positive real numbers such that
\begin{equation}
\label{eq:psiEst}
\Psi_{\text{min}}\leq \Psi_{i}(x)\leq\Psi_{\text{max}},\ \forall x\in\RR^d\text{ and }\Psi_{i}\in W^{2,\infty}\left(\mathbb{R}^{d} \right), \quad
\forall 1\leq i\leq r.
\end{equation}

\item\noindent{\bf  Assumptions on $R$}

\medskip

\begin{enumerate}
\item $R$ is continuous on $\RR^d\times\RR^r$.
\item There exists $A$ a positive real number such that\label{it:3}
\[
	-A\leq \partial_{v_{i}}R\left(x,v_{1},\dots,v_{r} \right)\leq -A^{-1},\quad \forall i \in \{1,\cdots,r\},\
        x\in\RR^d,\ v_1,\ldots, v_r\in \RR.
\]

\item There exist two positive constants $v_{\text{min}}<v_{\text{max}}$ such that \label{it:c}
\[
\min_{x\in\mathbb{R}^{d}}R(x,v)> 0 \text{ as soon as } \|v\|_{1}<v_{\text{min}},\text{ and }\max_{x\in\mathbb{R}^{d}}R(x,v)<0 \text{ as soon as } \|v\|_{1}>v_{\text{max}},
\]
where $\|v\|_{1}=|v_{1}|+\dots+|v_{r}|$.

\item Let $\mathcal{H}$ denotes the annulus $B\left(x,2v_{\text{max}} \right)\backslash B(x,v_{\text{min}}/2)$ (w.t.r. to the $\|\cdot\|_{1}$
  norm). Then\label{it:1}
\[
\sup_{v\in \mathcal{H}}\|R(\cdot,v)\|_{W^{2,\infty}}<M.
\]
\end{enumerate}
Note that the constant $2$ in the definition of $\mathcal{H}$ could be replaced by any constant strictly larger than $1$.

\item\noindent{\bf Assumptions on $h_\varepsilon$}

\medskip
\begin{enumerate}
\item $h_\varepsilon$ is Lipschitz-continuous on $\RR^d$, uniformly with respect to $\varepsilon>0$.
\item $h_\varepsilon$ converges in $L^\infty(\RR^d)$ as $\varepsilon\rightarrow 0$ to a function $h$.
\item For all $\varepsilon>0$ and all $1\leq i\leq d$,
  \begin{equation*}
    v_{\text{min}}\leq\int_{\RR^d} \Psi_i(x)\exp\left(-\frac{h_\varepsilon(x)}{\varepsilon}\right)dx\leq v_{\text{max}}.
  \end{equation*}
  In particular, $u_\varepsilon(0,x)$ is bounded in $L^1(\RR^d)$.\label{it:borne-h_eps}
\end{enumerate}
\end{enumerate}

Note that the limit $h$ of $h_\varepsilon$ is continuous, and hence, in order to satisfy Assumption~\ref{it:borne-h_eps}, it must
satisfy $h(x)\geq 0$ for all $x\in\RR^d$.

This type of assumptions is standard in this
domain~\cite{perthame-barles-08,barles-mirrahimi-al-09,lorz-mirrahimi-al-11,mirrahimi-roquejoffre-16}, but the previous references
only studied the case $r=1$. They are nearly identical (except for $h_\varepsilon$ and for linear or quadratic bounds at infinity,
see Corollary~\ref{cor:HJ-Var} in Section~\ref{sec:HJ-VAR}) as
those of~\cite{lorz-mirrahimi-al-11}. The case $r\geq 2$ was only studied in~\cite{champagnat-jabin-11}, but for a very specific form
of the function $R$, and in~\cite{perthame-barles-08} but without the convergence to the Hamilton-Jacobi problem.


\bigskip

Now, we present some preliminary results which are needed to study the asymptotic behaviour of the solution
of~\eqref{eq:prob}. The first result, Proposition~\ref{prop:prelimEstimates}, gives preliminary estimates on the solution
of~\eqref{eq:prob}. The second one, Theorem~\ref{thm:existence}, provides the existence and uniqueness of the solution of the
equation. These two results are direct adaptations of the results
of~\cite{perthame-barles-08,barles-mirrahimi-al-09,lorz-mirrahimi-al-11}. We detail the proof of the first one for sake of
completeness.

\begin{prop}(A priori estimates)
\label{prop:prelimEstimates}
Suppose that there exists a weak solution $u^{\ee}$ in $C\left(\mathbb{R}_{+};L^{1}\left(\mathbb{R}^{d} \right)\right)$. We
have, for all positive time $t$ and all $i\in\{1,\ldots,r\}$, 
 \[
v_{\textnormal{min}}-\frac{A\ee^{2}\Psi_{\textnormal{min}}}{\|\Psi_{i}\|_{W^{2,\infty}}}\leq v^{i,\ee}_{t}\leq v_{\textnormal{max}}+\frac{A\ee^{2}\Psi_{\textnormal{min}}}{\|\Psi_{i}\|_{W^{2,\infty}}},
\]
 \begin{equation}
 \label{eq:unifBound}
v_{\textnormal{min}}-\frac{A\ee^{2}\Psi_{\textnormal{min}}}{\inf_{1\leq i\leq r}\|\Psi_{i}\|_{W^{2,\infty}}}\leq \|v^{\ee}_{t} \|_{1}\leq v_{\textnormal{max}}+\frac{A\ee^{2}\Psi_{\textnormal{min}}}{\inf_{1\leq i\leq r}\|\Psi_{i}\|_{W^{2,\infty}}}
\end{equation}
and
 \begin{equation}
\label{eq:BoundU} 
\Psi_{\textnormal{max}}^{-1}\left(v_{\textnormal{min}}-\frac{A\ee^{2}\Psi_{\textnormal{min}}}{\inf_{1\leq i\leq r}\|\Psi_{i}\|_{W^{2,\infty}}}\right)\leq \int_{\mathbb{R}^{d}}u^{\ee}(t,x)dx\leq\left( v_{\textnormal{max}}+\frac{A\ee^{2}\Psi_{\textnormal{min}}}{\inf_{1\leq i\leq r}\|\Psi_{i}\|_{W^{2,\infty}}}\right)\Psi^{-1}_{\textnormal{min}}
\end{equation}
\end{prop}
\begin{proof}
Usual localization techniques for integration by parts entails, for all $1\leq i\leq r$,
\begin{equation*}
\partial_{t}v^{i,\ee}_{t}
=\ee\int_{\mathbb{R}^{d}}\Delta\Psi_{i}(x) u^{\ee}(t,x)dx+\frac{1}{\ee}\int_{\mathbb{R}^{d}}\Psi_{i}(x)u^{\ee}(t,x)R\left(x,v^{\ee}_{t} \right)dx.
\end{equation*}
Assumptions \eqref{eq:psiEst} on $\Psi_{i}$ leads to
 \begin{multline*}
-\frac{\ee\|\Psi_{i}\|_{W^{2,\infty}}}{\Psi_{\text{min}}}\int_{\mathbb{R}^{d}} \Psi_{i}(x) u^{\ee}(t,x)dx+\frac{1}{\ee}v^{i,\ee}_{t}\min_{x\in\mathbb{R}^{d}}R(x,v^{\ee}_{t})\leq\partial_{t}v^{i,\ee}_{t}\\
\leq\frac{\ee\|\Psi_{i}\|_{W^{2,\infty}}}{\Psi_{\text{min}}}\int_{\mathbb{R}^{d}}\Psi_{i}(x)  u^{\ee}(t,x)dx+\frac{1}{\ee}v^{i,\ee}_{t}\max_{x\in\mathbb{R}^{d}}R(x,v^{\ee}_{t}),
\end{multline*}
that is
 \begin{equation*}
-\frac{\ee\|\Psi_{i}\|_{W^{2,\infty}}}{\Psi_{\text{min}}}v^{i,\ee}_{t}+\frac{1}{\ee}v^{i,\ee}_{t}\min_{x\in\mathbb{R}^{d}}R(x,v^{\ee}_{t})\leq\partial_{t}v^{i,\ee}_{t}
\leq\frac{\ee\|\Psi_{i}\|_{W^{2,\infty}}}{\Psi_{\text{min}}}v^{i,\ee}_{t}+\frac{1}{\ee}v^{i,\ee}_{t}\max_{x\in\mathbb{R}^{d}}R(x,v^{\ee}_{t}).
\end{equation*}
Now, if $v^{i,\ee}_{t}>v_{\text{max}}+\frac{A\ee^{2}\Psi_{\text{min}}}{\|\Psi_{i}\|_{W^{2,\infty}}}$ (we recall that $M$ is the upper bound on the
derivatives of $R$ w.r.t.\ $v$), then $R(x,v^{\ee}_{t})<-\frac{\ee^{2}\Psi_{\text{min}}}{\|\Psi_{i}\|_{W^{2,\infty}}}$ for all $x$,  which
means that the derivative of $v^{i,\ee}_{t}$ becomes negative. Because of Assumption~\ref{it:borne-h_eps}, we deduce that
\[
v^{i,\ee}_{t}\leq v_{\text{max}}+\frac{A\ee^{2}\Psi_{\text{min}}}{\|\Psi_{i}\|_{W^{2,\infty}}}.
\]
The second bound is obtained similarly and the last bound follows easily from the first one in conjunction with \eqref{eq:psiEst}.
\end{proof}
\begin{thm}
  \label{thm:existence}
  For $\varepsilon>0$ small enough, there exists a unique solution $u^{\ee}$ in $C^{1}(\mathbb{R}_{+},\ L^{1}(\mathbb{R}^{d}))$
  of~\eqref{eq:prob}.
\end{thm}
\begin{proof}
We do not detail the proof since it is a direct adaptation of the unidimensional case given in~\cite{barles-mirrahimi-al-09}. We
simply mention that, because of~\eqref{eq:BoundU},
\begin{equation}
\label{eq:prob2}
\partial_{t}u^{\ee}(t,x)=\ee \Delta u^{\ee}(t,x)+\frac{1}{\ee}u^{\ee}(t,x)\tilde{R}\left(x,v^\varepsilon_{t}\right)
\end{equation}
where
\[
\tilde{R}(x,v)=
\left\{
\begin{array}{ll}
R(x,v)& \text{ if } \norm{v}_{1}\in[v_{\text{min}}/2,2v_{\text{max}}],\\
R(x,2v_{\text{max}})& \text{ if } \norm{v}_{1}>2v_{\text{max}},\\
R(x,v_{\text{min}}/2)& \text{ if } \norm{v}_{1}<v_{\text{min}}/2.
\end{array}\right.
\]
Since $\tilde{R}$ is bounded, the non-linearity in~\eqref{eq:prob2} is Lipschitz, so we can use Theorem 3.13 in~\cite{perthame15} to
obtain the existence and uniqueness of a solution in $C([0,T],\ L^{1}(\mathbb{R}^{d}))$. 
\end{proof}

Since $R(x,v^\varepsilon_t)$ is a Lipschitz function, one can actually get higher regularity from the regularizing effect of
the Laplace operator. However, since we plan to extend our method to more general mutation operators, we shall only make use in the
sequel of the fact that $u^{\ee}\in C^{1}(\mathbb{R}_{+},\ L^{1}(\mathbb{R}^{d}))$.
\section{Feynman-Kac representation of the solution}
\label{sec:FK}

The purpose of this section is to prove the following integral representation of the solution of \eqref{eq:prob} through the
so-called Feynman-Kac formula. 
\begin{thm}(Feynman-Kac representation of the solution of \eqref{eq:prob})
\label{thm:FK}
Let $u^{\ee}$ be the unique weak solution of \eqref{eq:prob}, then
\begin{equation}
  \label{eq:FK-thm}
u^\varepsilon(t,x)=\mathbb{E}_{x}\left[\exp\left(-\frac{h_\varepsilon(X^\varepsilon_{t})}{\ee}+\frac{1}{\ee}\int_{0}^{t}R(X^\varepsilon_{t},v^{\ee}_{t-s})ds\right) \right],\quad \forall (t,x)\in \mathbb{R}_{+}\times\mathbb{R}^{d},
\end{equation}
where
for all $x\in\RR^d$, $\EE_x$ is the expectation associated to the probability measure $\PP_x$, under which $X^\varepsilon_0=x$
almost surely and the process $B_t=(X^\varepsilon_t-x)/\sqrt{\varepsilon}$ is 
a standard Brownian motion in $\RR^d$.
\end{thm}

Before proving this result, we recall usual notions of weak solutions to problem \eqref{eq:prob} (cf.\
e.g.~\cite{pazy,lunardi,engel,perthame15}). We say that a function $u$ in $C(\mathbb{R}_{+},L^{1}(\mathbb{R}^{d}))$ is a mild
solution of problem \eqref{eq:prob} if it satisfies the following integral equation
\begin{equation}
\label{eq:mild}
u(t,x)=P_{t}^{\ee}g_{\ee}(x)+\frac{1}{\ee}\int_{0}^{t}P^{\ee}_{t-s}\left(u(s,x)R(x,v^{\ee}_{s})\right)\ ds,
\end{equation}
where $g_\varepsilon(x)=\exp(-h_\varepsilon(x)/\varepsilon)$ and $(P^{\ee}_{t})_{t\in\mathbb{R}_{+}}$ is the standard heat semi-group defined on $L^{1}(\mathbb{R}^{d})$ by
\[
P^{\ee}_{t}f(x)=\int_{\mathbb{R}^{d}}\frac{1}{(2\pi t\ee )^{n/2}}e^{\frac{-\norm{x-y}^{2}}{2t\ee}}f(y)\ dy,\quad \forall f\in L^{1}(\mathbb{R}^{d}).
\]
We also say that a function $u$ in $C(\mathbb{R}_{+},L^{1}(\mathbb{R}^{d}))$ is a weak solution of problem \eqref{eq:prob} if for any compactly supported test function $\varphi$ of $C^{\infty}([0,\infty)\times \mathbb{R}^{d})$, we have
\[
\int_{\mathbb{R}_{+}\times\mathbb{R}^{d}}u(t,x)\left(-\partial_{t}\varphi(t,x)-\frac{\ee}{2}\Delta\varphi(t,x)\ \right) dx\, dt = \frac{1}{\ee}\int_{\mathbb{R}_{+}\times\mathbb{R}^{d}}u(t,x)R(x,v^{\ee}_{t})\varphi(t,x) dx\, dt +\int_{\mathbb{R}^{d}}g_{\ee}(x)\varphi(0,x)dx.
\]
We point out that the notion of weak solution given above is the one for which existence and uniqueness hold in Theorem~\ref{thm:existence}.
\begin{lem}
\label{lem:mild}
	Let $\bar{u}^{\ee}$ be defined as the right-hand-side of~\eqref{eq:FK-thm}. The function $\bar{u}^{\ee}$ belongs to $C(\mathbb{R}_{+},L^{1}(\mathbb{R}^{d}))$ and is a mild solution of problem \eqref{eq:prob}.
\end{lem}
\begin{proof}
	Let $t$ and $h$ be two positive real numbers.
	\begin{align*}
	&\norm{\bar{u}^{\ee}(t+h,\cdot)-\bar{u}^{\ee}(t,\cdot)}_{L^{1}(\mathbb{R}^{d})}\\&=\int_{\mathbb{R}^{d}}\Bigg|\mathbb{E}_{x}\Big[g_{\ee}(X^{\ee}_{t+h})\exp\left(\frac{1}{\ee}\int_{0}^{t}R(X^{\ee}_{s},\
          v^{\ee}_{t+h-s})\ ds\right)\exp\left(\frac{1}{\ee}\int_{0}^{h}R(X^{\ee}_{t+s},\ v^{\ee}_{h-s})\ ds\right) \\ & \qquad\qquad
        -g_{\ee}(X^{\ee}_t)\exp\left(\frac{1}{\ee}\int_{0}^{t}R(X^{\ee}_{s},\ v^{\ee}_{t-s})\ ds\right)\Big]\Bigg|\ dx\\
	&\leq \int_{\mathbb{R}^{d}}\mathbb{E}_{x}\left[\exp\left(\frac{1}{\ee}\int_{0}^{t}R(X^{\ee}_{s},\ v^{\ee}_{t+h-s})\ ds\right)\left|\exp\left(\frac{1}{\ee}\int_{0}^{h}R(X^{\ee}_{t+s},\ v^{\ee}_{h-s})\ ds\right)g_{\ee}(X^{\ee}_{t+h})-g_{\ee}(X^{\ee}_{t})\right|\right]\ dx\\
	& \phantom{\leq}+\int_{\mathbb{R}^{d}}\mathbb{E}_{x}\left[g_{\ee}(X^{\ee}_{t})\left|\exp\left(\frac{1}{\ee}\int_{0}^{t}R(X^{\ee}_{s},\ v^{\ee}_{t+h-s})\ ds\right)-\exp\left(\frac{1}{\ee}\int_{0}^{t}R(X^{\ee}_{s},\ v^{\ee}_{t-s})\ ds\right) \right|\right]\ dx.
	\end{align*}
	Using to Hypothesis \eqref{it:1}, we get 
	\begin{align*}
	\norm{\bar{u}^{\ee}(t+h,\cdot)-\bar{u}^{\ee}(t,\cdot)}_{L^{1}(\mathbb{R}^{d})}&\leq \int_{\mathbb{R}^{d}}e^{\frac{Mt}{\ee}}\left\{\mathbb{E}_{x}\Big[g_{\ee}(X^{\ee}_{t+h})\left(e^{\frac{Mh}{\ee}}-1\right)\Big]\ +\mathbb{E}_{x}\Big[\left|g_{\ee}(X^{\ee}_{t+h})-g_{\ee}(X^{\ee}_{t})\right|\Big]\ \right\} dx\\
	&+\frac{M}{\varepsilon} e^{\frac{Mt}{\varepsilon}}\int_{\mathbb{R}^{d}}\mathbb{E}_{x}\left[g_{\ee}(X^{\ee}_{t})\right]\int_{0}^{t}\left|v^{\ee}_{t+h-s}-v^{\ee}_{t-s}\right|ds \ dx.
	\end{align*}
        Since $(P_{t})_{t\in\mathbb{R}_{+}}$ preserves the $L^{1}(\mathbb{R}^{d})$
        norm, we obtain
	\begin{align*}
	\norm{\bar{u}^{\ee}(t+h,\cdot)-\bar{u}^{\ee}(t,\cdot)}_{L^{1}(\mathbb{R}^{d})}&\leq \norm{g_{\ee}}_{L^{1}(\mathbb{R}^{d})}e^{\frac{Mt}{\ee}}\left(e^{\frac{Mh}{\ee}}-1 +\frac{M}{\varepsilon}\int_{0}^{t}\left|v^{\ee}_{s+h}-v^{\ee}_{s}\right|\ ds \right)
	\\&+ e^{\frac{Mt}{\ee}} \mathbb{E}\left[\int_{\mathbb{R}^{d}}\left|P_{h}^{\ee}g_{\ee}(x+X^{\ee}_{t})-g_{\ee}(x+X^{\ee}_{t}) \right|\ dx \right].
	\end{align*}

	%
	Since $v^{\ee}$ is continuous and using the translational invariance of Lebesgue measure, this finally leads to
	\[
	\norm{\bar{u}^{\ee}(t+h,\cdot)-\bar{u}^{\ee}(t,\cdot)}_{L^{1}(\mathbb{R}^{d})}\leq o_{h}(1)+ e^{\frac{Mt}{\ee}}\norm{P_{h}^{\ee}g_{\ee}-g_{\ee}}_{L^{1}(\mathbb{R}^{d})}\xrightarrow[h\to0]{}0,
	\]
	which proves that $\bar{u}^{\ee}$ is in $C(\mathbb{R}_{+},L^{1}(\mathbb{R}^{d}))$.

	We now prove that $\bar{u}^{\ee}$ is a mild solution of problem \eqref{eq:prob}. First, Markov property gives 
	\begin{align}
	\mathbb{E}_{x}\biggl[g_{\ee}(X^{\ee}_{t}) & \exp\left(\int_{t-s}^{t}R(X^{\ee}_{\theta},\ v^{\ee}_{t-\theta})\ d\theta \right)
          R(X^{\ee}_{t-s},\ v^{\ee}_{s})\biggr] \notag \\ &
        =\mathbb{E}_{x}\left[\mathbb{E}_{X^{\ee}_{t-s}}\left[g_{\ee}(X^{\ee}_{s})\exp\left(\int_{0}^{s}R(X^{\ee}_{\theta},\
              v^{\ee}_{s-\theta})\ d\theta \right)\right] R(X^{\ee}_{t-s},\ v^{\ee}_{s})\right] \notag \\ & =\mathbb{E}_{x}\left[u(s,X^{\ee}_{t-s}) R(X^{\ee}_{t-s},\ v^{\ee}_{s})\right].
	\label{eq:oneEq?}
	\end{align}
	Now, we have that
	\[
	\bar{u}^{\ee}(t,x)=\mathbb{E}_{x}\left[g_{\ee}(X^{\ee}_t)\left(\exp\left(\frac{1}{\ee}\int_{0}^{t}R(X^{\ee}_{s},\ v^{\ee}_{t-s})\ ds\right)-1\right) \right]+\mathbb{E}_{x}\left[g_{\ee}(X^{\ee}_{t}) \right]
	\]
	and
	
	\[
	\exp\left(\int_{0}^{t}R(X^{\ee}_{s},\ v^{\ee}_{t-s})\ ds\right)-1=\int_{0}^{t}\exp\left(\int_{t-s}^{t}R(X^{\ee}_{u},\ v^{\ee}_{t-u})\ du \right) R(X^{\ee}_{t-s},\ v^{\ee}_{s})ds.
	\]
	Hence \eqref{eq:oneEq?} entails
	\[
	\bar{u}^{\ee}(t,x)=\int_{0}^{t}\mathbb{E}_{x}\left[\bar{u}^{\ee}(s,X^{\ee}_{t-s})R(X^{\ee}_{t-s},v^{\ee}_{s}) \right]ds+\mathbb{E}_{x}\left[g_{\ee}(X^{\ee}_{t}) \right]=P_{t}^{\ee}g_{\ee}(x)+\int_{0}^{t}P^{\ee}_{t-s}\left(\bar{u}^{\ee}(s,x)R(x,v^{\ee}_{s})\right)\ ds.
	\]
This ends the proof that $\bar{u}^{\ee}$ is a mild solution of problem \eqref{eq:prob}.
\end{proof}
\begin{lem}
	Let $u\in C(\mathbb{R}_{+},L^{1}(\mathbb{R}^{d}))$ be a mild solution of problem \eqref{eq:prob}, then $u$ is a weak solution of problem \eqref{eq:prob}.
\end{lem}
\begin{proof}
	We begin the proof by showing that the mild formulation \eqref{eq:mild} of problem \eqref{eq:prob} admits a unique solution
        in $C(\mathbb{R}_{+},L^{1}(\mathbb{R}^{d}))$. This is standard calculations for mild solutions (cf.\ e.g.~\cite{pazy,lunardi})
	Let $u$ and $v$ be two solutions lying in this space such that $u(0,\cdot)=v(0,\cdot)$. We have
	\begin{multline*}
          \norm{u(t,\cdot)-v(t,\cdot)}_{L^{1}(\mathbb{R}^{d})}
          =\left\|\int_{0}^{t}P^{\ee}_{t-s}\left[u(s,\cdot)R\left(\cdot,v^{\ee}_{s}\right)-v(s,\cdot)R\left(\cdot,v^{\ee}_{s}\right)
            \right]\ ds\right\|_{L^{1}(\mathbb{R}^{d})}\\ 
          \leq\int_{0}^{t}\norm{P^{\ee}_{t-s}\left[u(s,\cdot)R\left(\cdot,v^{\ee}_{s}\right)-v(s,\cdot)
              R\left(\cdot,v^{\ee}_{s}\right)\right]}_{L^{1}(\mathbb{R}^{d})}\ ds
	\end{multline*}
	Now, since $(P_{t}^{\ee})_{t\in\mathbb{R}_{+}}$ is a contractive semi-group in $L^{1}(\mathbb{R}^{d})$, we get, using also Assumption~\eqref{it:1},
	\[
	\norm{u(t,\cdot)-v(t,\cdot)}_{L^{1}(\mathbb{R}^{d})}\leq M\int_{0}^{t}\norm{u(s,\cdot)-v(s,\cdot)}_{L^{1}(\mathbb{R}^{d})}\ ds.
	\]
	This proves the assertion through Gronwall lemma.
	
	Now, assume that $u$ is the weak solution of problem \eqref{eq:prob}, and consider the following Cauchy problem
	\begin{equation}
	\left\{
	\begin{array}{ll}
	\partial_{t}w=\ee \Delta w+\frac{1}{\ee}u(t,x)R\left(x,v^{\ee}_{t}\right), &  \\ 
	w(0,x)=\exp\left(-\frac{h_\varepsilon(x)}{\ee} \right)=:g_{\ee}(x). & 
	\end{array} \right.
	\end{equation}
	Since the inhomogeneous term belongs to $L^{1}(\mathbb{R}^{d})$, it is well known that this problem  admits a unique weak solution which is given by Duhamel's formula:
	\[
	w(t,x)=P_{t}^{\ee}g_\varepsilon(x)+\frac{1}{\ee}\int_{0}^{t}P^{\ee}_{t-s}u(t,x)R\left(x,v^{\ee}_{s}\right)ds.
	\]
	Since $u$ is the solution of the above  problem, we obtain that \[
	u(t,x)=P_{t}^{\ee}g_\varepsilon(x)+\frac{1}{\ee}\int_{0}^{t}P^{\ee}_{t-s}u(t,x)R\left(x,v^{\ee}_{s}\right)ds.
	\]
	Hence, the lemma is proved by uniqueness of mild solutions for problem \eqref{eq:prob}.
\end{proof}

\section{Small diffusion asymptotic}
\label{sec:Varadhan}

Recall from Theorem \ref{thm:FK} that
\begin{equation}
\label{eq:FK}
u^{\ee}\left(t,x \right)=\mathbb{E}_{x}\left[\exp\left(-\ee^{-1}h\left(X^{\ee}_{t}\right)+\frac{1}{\ee}\int_{0}^{t}R\left(X^{\ee}_{s},~ v^{\ee}_{t-s} \right)ds \right) \right],
\end{equation}
where $X^\varepsilon_t=x+\sqrt{\varepsilon}B_t$ under $\PP_x$. This formula suggests to apply Varadhan's lemma to study the
convergence of $\varepsilon\log u^\varepsilon(t,x)$ as $\varepsilon\rightarrow 0$. Let us fix $t>0$.
\begin{lem}
\label{lem:unifContinuity}

The function $\Phi_{\ee}:C\left([0,t]\right)\to\mathbb{R} $ defined by 
\[
\Phi_{\ee}\left(\varphi \right)=\int_{0}^{t}R\left(\varphi_{s},v^{\ee}_{s}\right)ds
\]
 is Lipschitz continuous on $C([0,t])$ endowed with the $L^{\infty}$-norm, uniformly w.t.r.\ to $\ee$ for $\varepsilon$
 small  enough.

 Moreover, there exists a kernel $\mathcal{M}$ on $\mathbb{R}_{+}\times\mathcal{B}(\mathbb{R}^{k})$ such that, along a subsequence $\left(\ee_{k}\right)_{k\geq 1}$ converging to $0$, we have
 \[
  \Phi\left(\varphi \right):=\lim\limits_{k\to\infty}\Phi_{\ee_{k}}(\varphi)=\int_{0}^{t} \int_{\mathbb{R}^k}R(\varphi_{s},y)\ \mathcal{M}_{s}(dy)\ ds,\quad\forall\varphi\in C([0,t])
 \]
 We recall that a kernel $\mathcal{M}$ is a function from $\mathbb{R}_{+}\times \mathcal{B}(\mathbb{R}^{d})$ into $\mathbb{R}_{+}$ such that, for all $t\in\mathbb{R}_{+}$, $\mathcal{M}_{s}$ is a measure on $\mathcal{B}(\mathbb{R}^{d})$ and, for all $A\in\mathcal{B}(\mathbb{R}^{d})$, the function $s\to\mathcal{M}_{s}(A)$ is measurable.
\end{lem}
\begin{proof}
We begin by showing that $\Phi_{\ee}\left(\varphi \right)$ is continuous for all $\ee$, uniformly w.r.t. $\ee$. Since $R(\cdot,v)$
lies in $W^{2,\infty}$, it follows that it is Lipschitz continuous, uniformly for $v$ is the annulus ${\cal H}$ (see
Assumption~\ref{it:1}). Hence, for $\psi,\varphi\in C\left([0,t]\right)$, it follows from Morrey's inequality that
\[
\int_{0}^{t}|R(\varphi_{s},v^{\ee}_{s})-R(\psi_{s},v^{\ee}_{s}) |ds \leq \int_{0}^{t} \|R(\cdot,v^{\ee}_{s}) \|_{W^{1,\infty}}\|\varphi_{s}-\psi_{s} \|ds\leq t \sup_{s\in[0,t]}\|R(\cdot,v^{\ee}_{s}) \|_{W^{1,\infty}}\|\varphi-\psi \|_{L^{\infty}([0,t])}.
\] 
Hence, the result follows from \eqref{eq:unifBound}.

Now, fix $T>0$, and let $\Gamma^{\ee}_{T}$, for all $\ee$, be the measure defined on $\left([0,T]\times\mathcal{H},\mathcal{B}\left([0,T]\right)\otimes\mathcal{B}\left(\mathcal{H}\right) \right)$ by
\[
\Gamma^{\ee}_{T}\left(A\times B \right)=\int_{A}\mathbbm{1}_{v^{\ee}_{s}\in B}\ ds, \quad \forall A\in \mathcal{B}\left([0,T]\right),\ B\in \mathcal{B}\left(\mathcal{H}\right).
\]
Since $\left(\Gamma_{T}^{\ee} \right)_{\ee>0}$ is a family of finite measures defined on a compact metric space, it is weakly precompact. Hence, there exists a subsequence $\left(\ee_{k}^{T}\right)_{k\geq 1}$ such that, the sequence of measures $\left(\Gamma_{T}^{\ee_{k}^{T}} \right)_{k\geq 1}$ converges weakly to some measure denoted by $\Gamma_{T}$. It follows from a diagonal argument that there exists a sequence $\left(\ee_{k}\right)_{k\geq 1}$ such that, for any positive integer $n$, $\Gamma^{\ee_{k}}_{n}$ converges weakly to a measure $\Gamma_{n}$. Now, if one wants to define a measure on $\mathbb{R}_{+}$ using the family $\left(\Gamma_{n}\right)_{n\geq 1}$, he needs that for $m<n$, the restriction of $\Gamma_{n}$ on $[0,m]$ coincides with $\Gamma_{m}$.

To prove this we first remark that, for all $k\in\mathbb{N}$ and $\delta>0$,
\begin{equation}
\label{eq:measureOdelta}
\Gamma_{n}^{\ee_{k}}\left(\left(m-\delta,m+\delta \right)\times\mathcal{H} \right)=\int_{m-\delta}^{m+\delta}\mathbbm{1}_{v^{\ee_{k}}_{s}\in\mathcal{H}} ds=2\delta.
\end{equation}
Let $f$ be a bounded and continuous function on $[0,m]\times\mathcal{H}$ and $(f_{\ell})_{\ell\geq 1}$ a sequence of uniformly bounded continuous function on $[0,n]\times\mathcal{H}$ such that 
\[
f_{\ell}(t,x)\xrightarrow[\ell\to\infty]{}f(t,x)\mathbbm{1}_{t\in[0,m]},\quad \forall (t,x)\in[0,n]\times\mathcal{H},
\]
and $f_{\ell}=f$ on $[0,m-\delta)\times\mathcal{H}$ and $f_{l}=0$ on $(m+\delta,n]\times\mathcal{H}$.
Using \eqref{eq:measureOdelta} we have, for all $k,\ell\in\mathbb{N}$,
\[
\left|\Gamma_{m}^{\varepsilon_{k}}(f)-\Gamma_{n}^{\ee_{k}}(f_{\ell})\right|\leq c\delta,
\]
where the constant $c$ does not depend on $k$ and $\ell$ and $\mu(f)$ is the integral of a function $f$ w.r.t.\ a measure $\mu$ on the corresponding space. Since, in addition, 
$\Gamma_{m}^{\ee_{k}}(f)$ converges to $\Gamma_{m}(f)$, $\Gamma_{n}^{\ee_{k}}(f_{\ell})$ converges to $\Gamma_{n}(f_{\ell})$ when $k$ goes to infinity and, by Lebesgue's theorem, $\Gamma_{n}(f_{\ell})$ converges to $\Gamma_{n}(f)$ when $\ell$ goes to infinity,  we deduce
\[
\left| \Gamma_{m}(f)-\Gamma_{n}(f)\right|\leq c\delta.
\]
Since $\delta$ was arbitrary we have proved that  the restriction of $\Gamma_{n}$ on $[0,m]$ coincides with $\Gamma_{m}$. As a consequence, we can define on $\mathbb{R}_{+}\times\mathcal{H}$ a measure $\Gamma$ whose restrictions on $[0,n]\times\mathcal{H}$ is $\Gamma_{n}$ for all $n\in\mathbb{N}$. 

In particular, for all $t$, $\Gamma_{t}^{\ee_{k}}$ converges weakly to the restriction of $\Gamma$ to $[0,t]\times{\cal H}$. In addition, for all function $\varphi$ continuous on $[0,t]$, the map
$(s,x)\mapsto R(\varphi_{s},x)$ is continuous and hence 
 \[
 \Phi\left(\varphi \right):=\lim\limits_{k\to\infty}\int_{0}^{t}R\left(\varphi_{s},v^{\ee_{k}}_{s}\right)ds=\int_{[0,t]\times \mathbb{R}^{d}}R(\varphi_{s},x)\ \Gamma(ds,dx),\quad \forall t\leq T.
 \]

It remains to show that $\Gamma$ can be disintegrated along
$\left(\mathbb{R}_{+},\mathcal{B}\left(\mathbb{R}_{+}\right),\lambda\right)$. Let $A$ be a null set of $[0,T]$, for a fixed positive time $T$. It follows that there exists, for all  $\eta>0$, a denumerable family of open ball $\left(B^{\eta}_{n}\right)_{n\geq 1}$ such that
\[
A\subset\bigcup_{n=1}^{\infty}B_{n}^{\eta} \quad \text{and} \quad \lambda\left(\bigcup_{n=1}^{\infty}B_{n}^{\eta}\right)<\eta.
\] 
Let $H$ be a measurable set of $\mathcal{H}$, then
\[
\Gamma\left( A\times H \right)\leq \Gamma_{T}\left(\cup_{n\geq 1}B_{n}^{\eta}\times\mathcal{H}
\right)\leq\liminf_{k\to\infty}\Gamma_{T}^{\ee_{k}}\left(\bigcup_{n=1}^{\infty}B_{n}^{\eta}\times \mathcal{H} \right)< C\eta
\]
which implies that $\Gamma(A\times H)=0$.
According to Radon-Nikodym's theorem, there exists, for all $H$, an integrable function $s\to\mathcal{M}_{s}(H)$ such that 
\[
\Gamma\left( A\times H \right)=\int_{\mathbb{R}_{+}}\mathcal{M}_{s}(H)\ ds.
\]
Usual theory (cf.\ e.g.~\cite{bogachev}) ensures that there exists a modification of this map such that $H\in\mathcal{B}(\mathcal{H})\to \mathcal{M}_{s}(H)$ is a measure for almost all $s\in\mathbb{R}_{+}$.
\end{proof}

\begin{thm}
	\label{thm:varad}
For all $(t,x)$ in $\mathbb{R_{+}}\times\mathbb{R}^{d}$,
\begin{equation}
\label{eq:varia}
V(t,x):=\lim_{k\rightarrow\infty}~\ee_k\log u^{\ee_k}(t,x)=\sup_{\varphi\in \mathcal{G}_{t,x}}\left\{-h(\varphi_{0})+\Phi(\varphi)-I_t(\varphi)  \right\}
\end{equation}
with
\[
I_t\left(\varphi\right)=
\left\{
\begin{array}{ll}
\int_{0}^{t}\|\varphi'(s)\|^{2}ds & \text{if }\varphi\text{ is absolutely continuous,}\\
+\infty&\text{otherwise},
\end{array}
\right.
\]
$\mathcal{G}_{t,x}$ denotes the set of continuous functions from $[0,t]$ to $\RR^d$ such that $\varphi_{t}=x$, and $\Phi$ and
$(\varepsilon_k)_{k\geq 1}$ are defined in Lemma~\ref{lem:unifContinuity}.
\end{thm}
\begin{proof}

  The proof of Theorem \ref{thm:varad} is a simple adaptation of the classical proof of Varadhan's Lemma (cf.\
  e.g.~\cite{dembo-zeitouni}). The main difference is due to the dependence on $\ee$ of the functions $h_\varepsilon$ and
  $\Phi_{\ee}$ involved in formula \eqref{eq:FK}. We give the details for completeness and for future reference in
  Section~\ref{sec:discrete}. Let $\mu^{x}_{\epsilon}$ the law of $(X^{\ee}_{s},\ s\in[0,t])$, when $X_{0}^{\ee}=x$ almost surely on
  $C\left(\left[ 0,t\right] \right)$.
	
	Let $\varphi\in C([0,t])$ such that $\varphi_0=x$, and $\delta>0$. Since, $\Phi_{\ee}$ and $h_\varepsilon$ are both continuous uniformly w.r.t.\ $\ee$ and $h_{\ee}$ converges uniformly to $h$, there exists a neighborhood $G_{\varphi}$ of $\varphi$ in $C([0,t])$ such that, for $\ee$ small enough (independent of $\varphi$),
	\begin{equation}
	\label{eq:upperBound}
	\sup_{\psi\in \overline{G}_{\varphi}} -h_\varepsilon(\psi_{t})+\Phi_{\ee}(\hat{\psi})-\delta\leq
        h(\varphi_{t})+\Phi_{\ee}(\hat{\varphi})\leq \inf_{\psi\in G_{\varphi}} -h_\varepsilon(\psi_{t})+\Phi_{\ee}(\hat{\psi})+\delta,
	\end{equation}
        where we use the notation $\hat{\varphi}_s=\varphi_{t-s}$.
        In addition, Lemma \ref{lem:unifContinuity} implies that, for $k$ large enough,
	\[
	\Phi(\hat{\varphi})-\delta\leq \Phi_{\ee_{k}}(\hat{\varphi})\leq \Phi(\hat{\varphi})+\delta.
	\]

	We start with the lower bound: for $k$ large enough,
	\begin{align*}
	u^{\ee_{k}}(t,x)&\geq \mathbb{E}_{x}\left[\exp\left(-\frac{h_{\varepsilon_{k}}\left(X^{\ee_{k}}_{t}\right)}{\ee_{k}}+\frac{1}{\ee_{k}}\int_{0}^{t}R\left(X^{\ee_{k}}_{s},~ v^{\ee_{k}}_{t-s} \right)ds \right)\mathbbm{1}_{X^{\ee_{k}}\in G_{\varphi}} \right]\\
	&\geq \mathbb{P}_{x}\left(X^{\ee_{k}}\in G_\varphi\right)\exp\left(\frac{-h\left(\varphi_{t}\right)+\Phi\left(\hat{\varphi}
            \right)-3\delta}{\ee_{k}} \right),
	\end{align*}
	Schilder's theorem (large deviation principle for the Brownian motion~\cite{dembo-zeitouni}) gives  
	\[
	\liminf_{k\to\infty}\ee_k\log(u^{\ee_{k}}(t,x))\geq -h(\varphi_{t})+\Phi(\hat{\varphi})-I(\varphi)-3\delta=-h(\hat{\varphi}_{0})+\Phi(\hat{\varphi})-I(\hat{\varphi})-3\delta.
	\]
	Since this holds for all $\delta>0$, we have proved the lower bound of Varadhan's Lemma along the sequence $(\ee_{k})_{k\geq 1}$.

	For the upper bound, we use the fact that $I$ is a good rate function, i.e.\ it is lower semi-continuous with compact level
        sets. Let $a$ be such that
	\[
	a\leq \sup_{\varphi\in \mathcal{G}_{t,x}}\left\{-h(\varphi_{0})+\Phi(\varphi)-I(\varphi)  \right\}.
	\]
	The set
	\[
	K 
        =\{\varphi\in C([0,t]):\varphi_0=x,\
        I(\varphi)\leq Mt+\delta-a \}
	\]
        is compact, where the constant $M$ is defined in Assumption~\ref{it:1}. Hence there exists a finite family of functions
        $\varphi_{1},\dots,\varphi_{n}$ (for some integer $n$), such that $\varphi_i(0)=x$ for all $1\leq i\leq n$ and
	\[
	K\subset\bigcup_{i=1}^{n}G_{\varphi_{i}}.
	\]
	Using equation \eqref{eq:upperBound} and Lemma \ref{lem:unifContinuity}, we get 
	\begin{multline*}
	\limsup_{k\to\infty}\ee_{k}\log \mathbb{E}_{x}\left[\exp\left(-\frac{h_\varepsilon\left(X^{\ee_{k}}_{t}\right)}{\ee_{k}}+\frac{1}{\ee_{k}}\int_{0}^{t}R\left(X^{\ee_{k}}_{s},~ v^{\ee_{k}}_{t-s} \right)ds \right)\mathbbm{1}_{X^{\ee_{k}}\in G_{\varphi_{i}}} \right]\\\leq h(\varphi_{i}({t}))+\Phi(\hat{\varphi}_{i})-I(\varphi_{i})+2\delta
	\end{multline*}
	and, using the fact that $-h_\varepsilon\leq -h+\delta\leq \delta$ for $\varepsilon$ small enough,
	\begin{multline*}
	\limsup_{k\to\infty}\ee_{k}\log \mathbb{E}_{x}\left[\exp\left(-\frac{h_\varepsilon\left(X^{\ee_{k}}_{t}\right)}{\ee_{k}}+\frac{1}{\ee_{k}}\int_{0}^{t}R\left(X^{\ee_{k}}_{s},~ v^{\ee_{k}}_{t-s} \right)ds \right)\mathbbm{1}_{X^{\ee_{k}}\in K^{c}} \right]\\\leq Mt+\delta+\limsup_{k\to\infty}\ee_{k}\log(\mathbb{P}(X^{\ee_{k}}\notin K))\leq a.
	\end{multline*}
	Therefore
	\begin{align*}
	\limsup_{k\to\infty}\ee_{k}\log(u^{\ee_{k}}(t,x))&\leq \max\left\{a,\sup_{1\leq i\leq n} h(\hat{\varphi}_{i}({0}))+\Phi(\hat{\varphi}_{i})-I(\hat{\varphi}_{i})+2\delta\right\}\\&\leq \sup_{\varphi\in \mathcal{G}_{t,x}}\left\{h(\varphi_{0})+\Phi(\varphi)-I(\varphi)  \right\}+2\delta.
	\end{align*}
        This ends the proof of the upper bound.
	\end{proof}

	Thanks to the representation \eqref{eq:FK} of the solution of problem \eqref{eq:prob} the convergence given above can be enhanced.

	\begin{thm}
          \label{thm:main-Lipschitz}
          The convergence stated in Theorem \ref{thm:varad} holds uniformly on compact sets and the limit $V(t,x)$ is Lipschitz w.r.t.\
          $(t,x)\in\RR_+\times\RR^d$. 
	\end{thm}

This result is an immediate consequence of the following lemma.
\begin{lem}
\label{lem:liplogu}
  The function
  \[
  \left.
    \begin{array}{cll}
      \mathbb{R}_{+}\times\mathbb{R}^{d}&\to&\mathbb{R}\\
      (t,x)&\mapsto& \ee \log u^{\ee}(t,x),
    \end{array}
  \right.
  \]
  is Lipschitz w.r.t.\ $(t,x)\in\RR_+\times\RR^d$, uniformly w.r.t.\ $\ee$ (at least for $\ee$ small enough).
\end{lem}
	\begin{proof}
          Let $(t,x)$ and $(\delta,y)$ in $\mathbb{R}_{+}\times\mathbb{R}^{d}$. Using \eqref{eq:FK}, and writing
          $X^\varepsilon_t=X^\varepsilon_0+\sqrt{\varepsilon}B_t$ for a Brownian motion $B$, we get
		\begin{align*}
		u^\varepsilon(t+\delta,x+y)&=\mathbb{E}\left[\exp\left(-\frac{h_\varepsilon(x+y+\sqrt{\ee}B_{t+\delta})}{\ee}+\frac{1}{\ee}\int_{0}^{t+\delta}R(x+y+\sqrt{\ee}B_{s},v^{\ee}_{t+\delta-s})ds\right) \right].
	\end{align*}
	Now, Markov's property entails
	\begin{align*}
		u^\varepsilon(t+\delta,x+y)&=\mathbb{E}\Bigg[\exp\left(\frac{1}{\ee}\int_{0}^{\delta}R(x+y+\sqrt{\ee}B_{s},v^{\ee}_{t+\delta-s})ds\right)\\ &\qquad\times \mathbb{E}_{x+y+\sqrt{\ee}B_{\delta}}\left[\exp\left(-\frac{h_\varepsilon(X^{\varepsilon}_t)}{\ee}+\frac{1}{\ee}\int_{0}^{t}R(X^\varepsilon_s,v^{\ee}_{t-s})ds \right)\right]\Bigg]\\
		&=\mathbb{E}\Bigg[\exp\left(\frac{1}{\ee}\int_{0}^{\delta}R(x+y+\sqrt{\ee}B_{s},v^{\ee}_{t+\delta-s})ds\right) u^\varepsilon(t,x+y+\sqrt{\ee}B_{\delta}) \Bigg].
		\end{align*}
	Hence, using Assumption \ref{it:1}, we get
	\begin{equation}
	\label{eq:goodeq1}
	e^{-\frac{M \delta}{\ee} }\mathbb{E}\Bigg[ u^\varepsilon(t,x+y+\sqrt{\ee}B_{\delta}) \Bigg] \leq u^\varepsilon(t+\delta,x+y)\leq e^{\frac{M \delta}{\ee} }\mathbb{E}\Bigg[ u^\varepsilon(t,x+y+\sqrt{\ee}B_{\delta}) \Bigg].
	\end{equation}
	In addition, taking into account that $x\to R(x,v)$ is Lipschitz uniformly w.r.t.\ $v\in{\cal H}$ with Lipstchiz norm bounded by $M$, we have 
			\begin{align*}
		u^\varepsilon(t,x+y)&=\mathbb{E}\left[\exp\left(-\frac{h_\varepsilon(x+y+\sqrt{\ee}B_{t})}{\ee}+\frac{1}{\ee}\int_{0}^{t}R(x+y+\sqrt{\ee}B_{s},v^{\ee}_{t-s})ds\right) \right]\\	
		&\leq u^\varepsilon(t,x)\mathbb{E}\left[\exp\left(-\frac{h_\varepsilon(x+y+\sqrt{\ee}B_{t})-h_\varepsilon(x+\sqrt{\ee}B_{t})}{\ee}+\frac{tM\|y\|}{\ee}\right) \right],\\
		\end{align*}
		and
		\[
		u^\varepsilon(t,x+y)\geq  u^\varepsilon(t,x)\mathbb{E}\left[\exp\left(-\frac{h_\varepsilon(x+y+\sqrt{\ee}B_{t})-h_\varepsilon(x+\sqrt{\ee}B_{t})}{\ee}-\frac{tM\|y\|}{\ee}\right) \right].
		\]
		Since $h_\varepsilon$ is Lipstchiz uniformly w.r.t.\ $\varepsilon>0$,
		\[
		u^\varepsilon(t,x+y)\leq u^\varepsilon(t,x)e^{\frac{tM+\sup_\varepsilon\|h_\varepsilon\|_{\text{Lip}}}{\ee}\norm{y}}.
		\]
		Now, using \eqref{eq:goodeq1}, we get
		\[
		u^\varepsilon(t+\delta,x+y)\leq u^\varepsilon(t,x)e^{\frac{M \delta}{\ee}}\mathbb{E}\left[ e^{\frac{tM+\sup_\varepsilon\|h_\varepsilon\|_{\text{Lip}}}{\ee}\norm{y+\sqrt{\ee}B_{\delta}}}\right].
		\]
		Then, using the inequality $\norm{y+\sqrt{\ee}B_{\delta}}\leq\norm{y}+C\sqrt{\varepsilon}\sum_i|B^i_\delta|$ and
              elementary computations, we obtain
		\[
			u^\varepsilon(t+\delta,x+y)\leq u^\varepsilon(t,x)e^{\frac{M \delta}{\ee}}e^{\frac{tM+\sup_\varepsilon\|h_\varepsilon\|_{\text{Lip}}}{\ee}\left(\norm{y}+nC\delta \right)}.
		\]
		Hence,
		\[
		\ee\log u^\varepsilon(t+\delta,x+y)-\ee\log u^\varepsilon(t,x)\leq M\delta+\left(tM+\sup_\varepsilon\|h_\varepsilon\|_{\text{Lip}}\right)\left(\norm{y}+nC\delta \right) .
		\]
		The lower bound is obtained similarly and ends the proof.
		\end{proof}

                In the last results, one would like to be able to characterize the limit measure ${\cal M}_s$ in terms of the zeroes
                of $V(s,\cdot)$, in order to obtain a closed form of the optimization problem. Results on this question are known in
                models with a single resource ($r=1$)~\cite{perthame-barles-08,barles-mirrahimi-al-09,lorz-mirrahimi-al-11}, but the
                known results when $r\geq 2$~\cite{champagnat-jabin-11,champagnat-HDR} require stringent assumptions on the structure
                of the model, and indeed, it is possible to construct examples where there exist several measures ${\cal M}_s$
                satisfying the metastability condition of~\cite{champagnat-jabin-11}. Since our assumptions are more general, we
                cannot expect to obtain such results in full generality, so we will focus in Sections~\ref{sec:discrete}
                and~\ref{sec:uniqueness-E-finite} on a case where precise results can be obtained, granting uniqueness for the
                optimization problem~(\ref{eq:varia}).

\section{Extensions of Theorem~\ref{thm:varad}}
\label{sec:extensions}

The proof of Theorem~\ref{thm:varad} only makes use of few properties of the Brownian motion and of the Laplace operator used to model
mutations in~\eqref{eq:prob}. In particular, one expects that it may hold true for partial differential equations of the form
\begin{equation}
\label{eq:general-PDE}
\partial_{t}u^{\ee}(t,x)=L_\varepsilon u^{\ee}(t,x)+\frac{1}{\ee}u^{\ee}(t,x)R\left(x,v^{\ee}_{t}\right),\quad \forall t>0,\ x\in\RR^d,
\end{equation}
where $L_\varepsilon$ is a linear operator describing mutations.

For our approach to work in this situation, the probabilistic interpretation of Theorem~\ref{thm:FK} must extend to this
case, and Varadhan's lemma must be applied as in the proof of Theorem~\ref{thm:varad}. For this, one needs that
\begin{enumerate}
\item the operator $L_\varepsilon$ is the infinitesimal generator of a Markov process $(X^\varepsilon_t,t\geq 0)$;
\item existence and uniqueness of a weak solution hold for the partial differential equation~\eqref{eq:general-PDE} in an
  appropriate functional space, for example in $C^{1}(\mathbb{R}_{+},\ L^{1}(\mathbb{R}^{d}))$ as in Theorem~\ref{thm:existence}, and
  any $C(\mathbb{R}_{+},\ L^{1}(\mathbb{R}^{d}))$ mild solution to the PDE must be a weak solution;
\item the method used to prove the Feynman-Kac formula of Theorem~\ref{thm:FK} applies, in other words the function
  \begin{equation}
    \label{eq:regularite-a-verifier}
    \bar{u}^\varepsilon(t,x)=\mathbb{E}_x\left[\exp\left(-\frac{h_\varepsilon(X^\varepsilon_t)}{\varepsilon}+\frac{1}{\varepsilon}\int_0^t
        R(X^\varepsilon_s,v^\varepsilon_{t-s})ds\right)\right]
  \end{equation}
  can be shown to be $C(\mathbb{R}_{+},\ L^{1}(\mathbb{R}^{d}))$ (note that the proof that $\bar{u}$ is a mild solution to the PDE
  only relies on the Markov property, so it is true in general);
\item the family of Markov processes $(X^\varepsilon)_{\varepsilon>0}$ must satisfy a large deviations principle with rate
  $\varepsilon^{-1}$ and a good rate function, i.e.\ a lower semicontinuous function with compact level sets.
\end{enumerate}
Note that the compactness argument of Lemma~\ref{lem:unifContinuity} does not depend on the mutation operator. It only follows from
our assumptions on $R$.

For example, all these points apply to the problem
\begin{equation}
\label{eq:PDE-RW}
\partial_{t}u^{\ee}(t,x)=\frac{1}{\varepsilon}\int_{\mathbb{R}^n}\left[u^\varepsilon(t,x+\varepsilon z)-u^\varepsilon(t,x)\right]
K(z)dz+\frac{1}{\ee}u^{\ee}(t,x)R\left(x,v^{\ee}_{t}\right),
\end{equation}
where $K:\mathbb{R}^n\rightarrow \mathbb{R}_+$ is such that
\begin{equation*}
  \int_{\mathbb{R}^n}z K(z)dz=0\quad\text{and}\quad \int_{\mathbb{R}^n} e^{\langle a,z\rangle^{2}}K(z)dz<\infty\text{ for all
  }a\in\mathbb{R}^n.
\end{equation*}
This form of mutation operator has
already been studied in \cite{champagnat-jabin-11,barles-perthame-07,barles-mirrahimi-al-09}. Similar equations are also considered in \cite{jabin-raoul-11,raoul-11,desvillettes-jabin-al-08}. In this case,
we can check all the points above as follows:
\begin{enumerate}
\item The Markov process $X^\varepsilon$ is a continuous-time random walk, with jump rate $\frac{\|K\|_{L^1}}{\varepsilon}$ and
  i.i.d.\ jump steps distributed as $\varepsilon Z$, where the random variable $Z$ has law $\frac{K(z)}{\|K\|_{L^1}}dz$.
\item Existence and uniqueness of the solution to~\eqref{eq:PDE-RW} in $C^{1}(\mathbb{R}_{+},\ L^{1}(\mathbb{R}^{d}))$ follows from
  \cite{barles-mirrahimi-al-09} (this is the point where the finiteness of quadratic exponential moments of $K$ is needed).
\item Since random walks are shift-invariant as Brownian motion, the regularity of~\eqref{eq:regularite-a-verifier} can be proved
  exactly as in the proof of Theorem~\ref{thm:FK}.
\item It is well-known (see \cite[Section 10.3]{LDPKurtz}) that, under the previous assumptions,
  the family $(X^\varepsilon)_{\varepsilon>0}$ satisfies a large deviation principle with rate $\varepsilon^{-1}$ and good rate
  function on $\mathbb{D}([0,t],\mathbb{R}^n)$, the set of c\`adl\`ag functions from $[0,t]$ to $\mathbb{R}^n$, given by
  \[
  I_t(\varphi)=
  \begin{cases}
    \int_0^t\int_{\RR^d}\left(e^{ \langle z,\dot{\varphi}_s\rangle}-1\right) K(z)dz\,ds & \text{if $\varphi$ is absolutely continuous,} \\
    +\infty & \text{otherwise.}
  \end{cases}
  \]
\end{enumerate}

\medskip

We will give in Section~\ref{sec:discrete} another example of extension, in the case where the Markov processes $X^\varepsilon$ take
values in a finite set. However, this case will raise specific difficulties because the rate function of the associated large deviations
principle has non-compact level sets.

\section{Links between the variational and Hamiltom-Jacobi problems}
\label{sec:HJ-VAR}

As explained in the introduction, the convergence of $\ee\log u^{\ee}$ has been studied in various works using PDE approaches, with a
limit solving a Hamilton-Jacobi problem with constraints.
With our approach, instead of a Hamilton-Jacobi equation, we obtain a variational characterization of the limit, under assumptions on
$R$ slightly weaker than those of~\cite{perthame-barles-08,barles-mirrahimi-al-09,lorz-mirrahimi-al-11} and valid for any value of
the parameter $r$, biologically interpreted as a number of resources (see the introduction). Therefore, we obtain naturally the
identification between a solution to Hamilton-Jacobi problems and a variational problem. For example, the next result is a
consequence of~\cite[Eq.\,(5.8)]{lorz-mirrahimi-al-11}.

\begin{cor}
  \label{cor:HJ-Var}
  In addition to the assumptions of Section~\ref{sec:problem}, assume that $r=1$, $h_\varepsilon$ is $C^2$ uniformly in
  $\varepsilon>0$ and there exist constants $A,B,C,D$ such that, for all $x\in\RR^d$, $v_{\textnormal{min}}\leq \|v\|_1\leq
  v_{\textnormal{max}}$ and $\varepsilon>0$,
  \begin{gather*}
    -A|x|^2\leq R(x,v)\leq B-A^{-1}|x|^2, \\
    -B+A^{-1}|x|^2\leq h_\varepsilon(x)\leq B+A|x|^2, \\
    -C\textnormal{Id}\leq D^2 R(x,v)\leq -C^{-1}\textnormal{Id}, \\
    C^{-1}\textnormal{Id}\leq D^2 h_\varepsilon(x)\leq C\textnormal{Id}, \\
    \Delta(\Psi_1 R)\geq -D,
  \end{gather*}
  where the third oud fourth inequalities have to be understood in the sense of symmetric matrices and $\textnormal{Id}$ is the
  $n$-dimensional identity matrix. Then, $v^{\varepsilon_k}$ converges in $L^1_{\textnormal{loc}}(\RR_+)$ to a nondecreasing limit
  $\bar{v}$ along the subsequence $\varepsilon_k$ of Lemma~\ref{lem:unifContinuity}, the kernel ${\cal M}$ satisfies
  \begin{equation}
    \label{eq:expr-M-HJ-Var}
    {\cal M}_s(dy)=\delta_{\bar{v}(s)}(dy),\quad\forall s\geq 0,    
  \end{equation}
  and the limit $V$ of Theorem~\ref{thm:varad} solves in the viscosity sense
  \begin{equation}
    \label{eq:HJ-Lorz}
    \begin{cases}
      \partial_t V(t,x)=R(x,\bar{v}_t)+\frac{1}{2}|\nabla V(t,x)|,\quad\forall t\geq 0,\ x\in\RR^d, \\
      \max_{x\in\RR^d}V(t,x)=0,\quad\forall t\geq 0.
    \end{cases}
  \end{equation}
\end{cor}

The only point of this corollary that is not a direct consequence of~\cite{lorz-mirrahimi-al-11} is the convergence of the full
sequence $(v^{\varepsilon_k})_{k\geq 1}$ to $\bar{v}$. We known from~\cite{lorz-mirrahimi-al-11} that such a convergence is true for
a subsequence and hence~\eqref{eq:expr-M-HJ-Var} holds true. Since, again from~\cite{lorz-mirrahimi-al-11}, any subsequence of
$v^{\varepsilon_k}$ admits a subsequence converging in $L^1$ and any accumulation point in $L^1$ must be $\bar{v}$
by~\eqref{eq:expr-M-HJ-Var}, the result is clear.

The last corollary provides a variational interpretation~\eqref{eq:varia} for the solution to~\eqref{eq:HJ-Lorz} given by the limit
of $\varepsilon_k\log u^{\varepsilon_k}$. Since uniqueness for~\eqref{eq:varia} is not known in general (note however that it is
known under additional assumptions, see~\cite{mirrahimi-roquejoffre-16}), we cannot deduce that any solution to~\eqref{eq:HJ-Lorz}
admits a variational formulation.

Note that Hamilton-Jacobi equations are known to be related to variational problems appearing in control theory. However, this link
is known in general in cases with continuous and time-independent coefficients in the Hamilton-Jacobi
equation~\cite{lions-82,fleming-soner-93}. In our case, it is known that $\bar{v}$ may be discontinuous in general, hence
Corollary~\ref{cor:HJ-Var} cannot be deduced from these general results. Hamilton-Jacobi problems with discontinuous coefficients
were studied in several works but none covers the form of the optimization problem given by our results: first, the question of
appropriate notion of viscosity solutions of irregular Hamilton-Jacobi problems has been studied by several authors starting
from~\cite{perthame-lions,bardi-capuzzo}. In the later reference, these notions are only used to study variational problems
without integral term and with continuous Hamiltonian without time-dependencies. In the book~\cite{clarke}, problems of calculus of
variation with irregular Lagrangian are studied, providing existence of optimal solutions but no link with the Hamilton-Jacobi
problem. Finally, in the article~\cite{vinter-wolenski}, optimal control problems with measurable in time Hamiltonian are studied but
without the integral term in the optimisation problem as in~\eqref{eq:varia}. Hence Corollary~\ref{cor:HJ-Var} seems to give, through
an indirect approach, an original optimization formulation for Hamilton-Jacobi problems with constraints as those
of~\cite{diekmann-jabin-al-05,lorz-mirrahimi-al-11}.

Other mutation operators were also studied with the PDE approach. Up to our knowledge, the only work providing the Hamilton-Jacobi
limit with several resources is~\cite{champagnat-jabin-11}. In this case, we obtain the following corollary.

\begin{cor}
  \label{cor:HJ-Var-chemostat}
  Assume that $d=1$, $r\geq 1$ is arbitrary and $R$ is given by Equation~\eqref{eq:chemostat} with $d(x)\equiv 1$. Consider
  the solution $u^\varepsilon$ of the PDE~\eqref{eq:PDE-RW} with this function $R$ and with a function $K$ in $C^\infty_c$ and such
  that $\int_{\RR}zK(z)\,dz=0$. Assume that the assumptions of Section~\ref{sec:problem} are satisfied, that $\Psi_i$ is $C^2_b$
  for all $1\leq i\leq r$, that $\sum_{i=1}^r c_i\Psi_i(x)\geq 1$ for all $x\in\RR$ and
  \begin{gather*}
    \exists\bar{r}\leq r,\ \forall v_1\in[0,c_1],\ldots,v_r\in[0,c_r],\text{ the function }x\mapsto\sum_{i=1}^{r}v_i\Psi_i(x)-1\text{
      has at most }\bar{r}\text{ roots}, \\
    \forall x_1,\ldots,x_{\bar{r}}\text{ distinct, the $\bar{r}$ vectors }\Psi(x_1),\ldots,\Psi(x_{\bar{r}})\text{ are linearly independent},
  \end{gather*}
  where $\Psi(x)=(\Psi_1(x),\ldots,\Psi_r(x))$. Assume also that $h_\varepsilon$ is $C^2$ for all $\varepsilon>0$, $h_\varepsilon-h$ converges
  to $0$ in $W^{1,\infty}(\RR)$,
  \begin{gather*}
    \sup_{\varepsilon>0}\left\|\partial_x h_\varepsilon\right\|_{L^\infty(\RR)}<\infty\quad\text{and}\quad
    \inf_{\varepsilon>0}\inf_{x\in\RR}\partial_{xx}h_\varepsilon(x)>-\infty.
  \end{gather*}
  Then, $v^{\varepsilon_k}$ converges in $L^1_{\textnormal{loc}}$ to $\bar{v}$ along the subsequence $\varepsilon_k$ of
  Lemma~\ref{lem:unifContinuity}, where $\bar{v}(t)=F(\{V(t,\cdot)=0\})$ for an explicit function $F$
  (see~\cite[Prop.\,1.1]{champagnat-jabin-11}), the kernel ${\cal M}$ satisfies
  \begin{equation*}
    {\cal M}_s(dy)=\delta_{\bar{v}(s)}(dy),\quad\forall s\geq 0,    
  \end{equation*}
  and the limit $V$ of Theorem~\ref{thm:varad} is such that 
  $\varphi(t,x)=V(t,x)-\sum_{i=1}^r\int_0^t \bar{v}_i(s)ds\,\Psi_i(x)$ solves in the viscosity sense
  \begin{equation*}
    \begin{cases}
      \partial_t \varphi(t,x)=H\left(\partial_x\varphi(t,x)+\sum_{i=1}^r\int_0^t \bar{v}_i(s)ds\,\Psi_i(x)\right),\quad\forall t\geq 0,\ x\in\RR^d, \\
      \max_{x\in\RR^d} \left\{\varphi(t,x)+\sum_{i=1}^r\int_0^t \bar{v}_i(s)ds\,\Psi_i(x)\right\}=0,\quad\forall t\geq 0,
    \end{cases}
  \end{equation*}
  where
  \[
  H(p)=\int_{\RR}(e^{pz}-1)K(z)\,dz.
  \]
\end{cor}

\section{The finite case}
\label{sec:discrete}

We consider here another extension of Theorem~\ref{thm:varad} to processes satisfying a large deviation principle, in the case where
the trait space if finite. Since in this case the equation satisfied by the limit $V$ is simpler, we are able to study it in details
in Section~\ref{sec:uniqueness-E-finite}. We consider a finite set $E$ and the system of ordinary differential equations
\begin{equation}
\label{eq:EDO}
\left\{
\begin{array}{l}
\dot{u}^{\ee}(t,i)=\sum_{j\in E\backslash\{i
  \}}\exp\left(-\frac{\mathfrak{T}(i,j)}{\ee}\right)(u^{\ee}(t,j)-u^{\ee}(t,i))+\frac{1}{\ee}u^{\ee}(t,i)R(i,v^{\ee}_{t}),\quad
\forall t\in[0,T],\ \forall i\in E\\
u^{\ee}(0,i)=\exp\left(-\frac{h^{\ee}(i)}{\ee}\right),
\end{array}
\right.
\end{equation}
where $\mathfrak{T}(i,j)\in(0,+\infty]$ for all $i\neq j\in E$, $R(i,u):E\times \mathbb{R}^r\mapsto \mathbb{R}$,
$h^\varepsilon:E\rightarrow\RR$ and $v^{\ee}_{t}=(v^{1,\varepsilon}_t,\ldots,v^{r,\varepsilon})$ is defined by
\[
v^{i,\ee}_{t}=\sum_{j\in E}u(t,j)\Psi_{i}(j), \quad \forall 1\leq i\leq r,
\]
for some functions $\Psi_{i}:E\mapsto(0,+\infty)$. The term $e^{-\varepsilon^{-1}\mathfrak{T}(i,j)}$ corresponds to the mutation rate
from trait $i$ to trait $j$. Its value is by convention $0$ when $\mathfrak{T}(i,j)=+\infty$, which means than mutations are
impossible from state $i$ to state $j$.

The standing assumptions on $\Psi$, $R$ and $h_\varepsilon$ given in Section~\ref{sec:problem} are still assumed to hold true here
replacing $\RR^d$ by $E$ (except of course for the assumptions of regularity in the trait space). We also need the following assumption.
\medskip
  
\noindent{\bf  Assumption on $\mathfrak{T}$}

For all distinct $i,j,k\in E$,
\begin{equation}
\label{hyp:jumpMet}
\mathfrak{T}(i,j)+\mathfrak{T}(j,k)>\mathfrak{T}(i,k)
\end{equation}
(with the convention $\infty>\infty$) and we set
\begin{equation}
\label{eq:hypTij}
\eta:=\inf_{i,j,k\in E \, \text{distinct s.t.\ }\mathfrak{T}(i,k)<\infty}\mathfrak{T}(i,j)+\mathfrak{T}(j,k)-\mathfrak{T}(i,k)>0.
\end{equation}

\medskip
Similarly as for problem \eqref{eq:EDP}, the solution of problem \eqref{eq:EDO} remains bounded. This is given in the following Lemma.
 \begin{lem}
 \label{lem:upBound}
 We have for all $t\geq 0$
 \[
 \frac{v_{\textnormal{min}}-A^{-1}(|E|-1)e^{-\beta/\varepsilon}}{\Psi_{\textnormal{max}}}\leq \sum_{i\in E}u^{\ee}(t,i)\leq \frac{v_{\textnormal{max}}+A(|E|-1)e^{-\gamma/\varepsilon}}{\Psi_{\textnormal{min}}}
 \]
 holds true for any positive $t$ as far as it holds for $t=0$, where $|E|$ stands for the cardinality of $E$,
 \[
 \gamma=\inf\{\mathfrak{T}(i,j)\mid i,j\in E,\ i\neq j \}>0,\quad  \beta=\sup\{\mathfrak{T}(i,j)\mid i,j\in E,\ i\neq j \}
 \]
 and the constants $A$, $v_{\textnormal{min}}$, $v_{\textnormal{max}}$, $\Psi_{\textnormal{min}}$ and $\Psi_{\textnormal{max}}$ are
 defined in Section~\ref{sec:problem}.
 \end{lem}

Hence we shall also assume in the sequel that
\[
\frac{v_{\textnormal{min}}-A^{-1}(|E|-1)e^{-\beta/\varepsilon}}{\Psi_{\textnormal{max}}}\leq \sum_{i\in E}e^{-h_\varepsilon(i)/\varepsilon}\leq
\frac{v_{\textnormal{max}}+A(|E|-1)e^{-\gamma/\varepsilon}}{\Psi_{\textnormal{min}}}.
\]
We also define the compact set
\[
{\cal H}:=\left\{ u\in\RR_+^E : \frac{v_{\textnormal{min}}-A^{-1}(|E|-1)e^{-\beta/\varepsilon}}{\Psi_{\textnormal{max}}}\leq
  \|u\|_1\leq
\frac{v_{\textnormal{max}}+A(|E|-1)e^{-\gamma/\varepsilon}}{\Psi_{\textnormal{min}}}
\right\}
\]
which is invariant for the dynamics~\eqref{eq:EDO}.

\begin{proof}
 Set, for any positive real number $t$,
 \[
 u^{\ee}(t)=\sum_{i\in E}u^{\ee}(t,i).
 \]
 According to \eqref{eq:EDO}, we get
 \[
 \dot{u}^{\ee}(t)\leq u^{\ee}(t) \, \max_{i\in E}R(i,v^{\ee}_{t})+\sum_{i\in E}\sum_{j\in E\setminus\{ i\}}\exp\left({-\frac{\mathfrak{T}(i,j)}{\ee}}\right)(u^{\ee}(t,j)-u^{\ee}(t,i)).
 \]
 Hence
 \[
 \dot{u}^{\ee}(t)\leq u^{\ee}(t) \left( \max_{i\in E}R(i,v^{\ee}_{t})+(\left|E\right|-1)e^{-\frac{\gamma}{\ee}}\right).
 \]
 Now, using Hypothesis \eqref{it:3} and \eqref{it:c}, we obtain that
 \[
 \max_{i\in E}R(i,v^{\ee}_{t})<-(\left|E\right|-1)e^{-\frac{\gamma}{\ee}}
 \]
 as soon as $v^{i,\ee}_{t}>v_{\text{max}}+A(\left|E\right|-1)e^{-\frac{\gamma}{\ee}}$, for some $i$ in $\{1,\dots,r\}$.
 Using Hypothesis \eqref{eq:psiEst}, we deduce that
 \[
 u^\varepsilon(t)\leq \frac{1}{\Psi_{min}}\left(v_{\text{max}}+ A(\left|E\right|-1)e^{-\frac{\gamma}{\ee}}\right)
 \]
 as soon as this holds at time $t=0$. The proof of the lower bound is similar.
 \end{proof}
Our first goal is to describe the solution $u^{\ee}$ of the system using an integral representation similar to~\eqref{eq:FK}.
Let $(X^{\ee}_{s},\ s\in[0,T])$ be the Markov processes in $E$ with infinitesimal generator
\[
L^{\ee}f(i)=\sum_{j\in E}(f(j)-f(i))e^{-\frac{\mathfrak{T}(i,j)}{\varepsilon}},
\]
i.e.\ the Markov process which jumps from state $i\in E$ to $j\neq i$ with rate 
$\exp(-\mathfrak{T}(i,j)/\ee)$.
 
\begin{prop} (Integral representation) For any positive real number $t$ and any element $i$ of $E$, we have
\[
u^{\ee}\left(t,i\right)=\mathbb{E}_{i}\left[\exp\left(-\frac{h^{\ee}(X^{\ee}_{t})}{\ee}+\frac{1}{\ee}\int_{0}^{t}R\left(X^{\ee}_{s},~ v^{\ee}_{t-s} \right)ds \right) \right].
\]
\end{prop}
\begin{proof}
First note that the part of the proof of Lemma \ref{lem:mild} showing that \eqref{eq:FK} is a mild solution of problem \eqref{eq:prob} do not rely in any manner on the Brownian nature of $B^{\ee}$. As a consequence, one can directly deduce that $u^{\ee}\left(t,i\right)$ statifies
	\[
	u^{\ee}(t,i)=P_{t}^{\ee}(e^{-h^\ee/\ee})(i)+\int_{0}^{t}P^{\ee}_{t-s}\left(u^{\ee}(s,i)R(i,v^{\ee}_{s})\right)\ ds,
	\]
	where $(P^{\ee}_{t},\ t\in\mathbb{R}_{+})$ now stands for the semigroup generated by $L^{\ee}$ (i.e.\ a simple exponential of
        matrix since $E$ is finite). This last expression is the
        Duhamel formulation of~\eqref{eq:EDO}. 
\end{proof}
The next result proves a weak large deviations principle  (i.e.\ a large deviations principle with upper bounds only for compact sets of $\mathbb{D}([0,T],E)$) for the laws of $X^{\ee}$.
\begin{prop} (Weak LDP)
	\label{prop:LDP}
$(X^{\ee})_{\ee\geq0}$ satisfies a weak LDP with rate function 
\[
\left.
\begin{array}{lcll}
I_{T}:&\mathbb{D}([0,T],E)&\mapsto&\mathbb{R}\\
&\varphi&\to& \sum_{l=1}^{N_{\varphi}}\mathfrak{T}\left(\varphi_{t^\varphi_{l}-},\varphi_{t^\varphi_{l}}\right),
\end{array}
\right.
\]
where $\mathbb{D}([0,T],E)$ is the space of c\`adl\`ag functions from $[0,T]$ to $E$ and $N_{\varphi}$ is the number of jumps of
$\varphi$ and $(t^\varphi_{l})_{1\leq l\leq N_{\varphi}}$ the increasing sequence of jump times of $\varphi$.
\end{prop}

We shall also make use of the notation
\[
I_T(\varphi)=\sum_{0<s\leq T}\mathfrak{T}(\varphi_{s-},\varphi_s)
\]
with the implicit convention that $\mathfrak{T}(i,i)=0$ for all $i\in E$.

Before proving this result, we focus on the following lemma which provides a convenient topological basis of the space
$\mathbb{D}([0,T],E)$ equipped with the Skorohod topology.
\begin{lem}
	\label{lem:sko}
        For all $\varphi\in \mathbb{D}([0,T],E)$ and $\delta<1$ define
	 \[
	 B_{\text{Sko}}(\varphi,\delta)=\{\psi\in\mathbb{D}([0,T],E)\mid N_{\varphi}=N_{\psi},\ |t^{\varphi}_{i}-t^{\psi}_{i}|<\delta,\ \varphi_0=\psi_0, \text{ and } \varphi_{t_{i}^{\varphi}}=\psi_{t_{i}^{\psi}}. \}
	 \]
	Then,
	the set
	\[
	\left\{B_{\text{Sko}}(\varphi,\ee)\mid \ee\in[0,1),\ \varphi\in\mathbb{D}([0,T],E) \right\}
	\]
	is a topological basis of $\mathbb{D}([0,T],E)$.
\end{lem}
\begin{proof}
	To prove the result it is enough to show that, for all $\delta<1$ and $\varphi\in\mathbb{D}([0,T],E)$ the set
        $B_{\text{Sko}}(\varphi,\delta)$ is exactly the $\delta$ neighborhood of $\varphi$ for a particular metric inducing the Skorokhod topology.
	We recall (see e.g.\ \cite{Billingsley}) that the Skorokhod topology can be defined through the metric $d_{S}$ given by
	\[
	d_{S}(\varphi,\psi)=\inf_{\lambda\in \Lambda}\left\{ \max\left(\|\lambda-I \|_{L^\infty([0,T])},\sup_{t\in [0,T]} d(\varphi_t,\psi \circ \lambda_t)\right) \right\}
	\]
	where $\Lambda$ is the set of continuous increasing functions on $[0,T]$ with $\lambda_0=0$ and $\lambda_T=T$ and the
        distance $d$ on $E$ is defined as $d(i,j)=\mathbbm{1}_{i\neq j}$.
	Let $\psi$ and $\varphi$ be such that $d_{S}(\varphi,\psi)<\delta$ for some $\delta<1$, then there exists $\lambda$ in $\Lambda$ such that
	\[
	\left\{
	\begin{array}{l}
	|\lambda_{s}-s|<\delta, \quad \forall s\in[0,T],\\
	d(\varphi_{s},\psi(\lambda_{s}))<\delta,\quad \forall s\in[0,T].
	\end{array}
	\right.
	\]
	Since $\inf_{i,j\in E,\ i\neq j}d(i,j)=1$, we have
	\[
	\varphi_{t}=\psi(\lambda_{t}),\quad \forall t\in[0,T].
	\]
Hence, it follows that $N_{\varphi}$ equals $N_{\psi}$ and $t^\varphi_{i}=\lambda(t^{\psi}_{i})$, for all $i$. Consequently, $|t^{\varphi}_{i}-t^{\psi}_{i}|<\delta$, for all $i$, and $\varphi_{t^\varphi_{i}}=\psi_{t^\psi_{i}}$.
\end{proof}
We can now prove the weak large deviations principle for $X^\ee$.
\begin{proof}[Proof of Proposition \ref{prop:LDP}]
  Let $(S_{i})_{i\geq0}$ be the discrete time Markov chain associated to $X^{\ee}$ and let $T_i$ be the $i$-th inter-jump time, i.e.
  $T_i=J_{i+1}-J_i$ for all $i\geq 0$, where $J_i$ is the $i$-th jump time of $X^{\varepsilon}$, for $i\geq 1$ and $J_0=0$. We recall that the Markov
  chain $(T_i,S_i)_{i\geq 0}$ has
  transition kernel given by
\[
P((t,i),ds,dj)=\exp\left(-\frac{\mathfrak{T}(i,j)}{\ee}\right)\exp(-c^{\ee}(i)s)\ ds\otimes C(dj),\quad \forall(t,i)\in [0,T]\times E,
\] 
where
\[
c^\varepsilon(i)=\sum_{j\in E,\ j\neq i}\exp\left(-\frac{\mathfrak{T}(i,j)}{\varepsilon}\right)
\]
and $C(dj)$ is the counting measure on $E$
and initial value $(T_{0},S_{0})=(0,i)$ $\mathbb{P}_{i}$-almost surely.

\medskip

Let $\varphi$ be an element of $\mathbb{D}([0,T],E)$ such that $\varphi(0)=i$ and $\delta<1$. We set $t^\varphi_{0}=0$. According to
Lemma \ref{lem:sko},
\begin{multline*}
\mathbb{P}_{i}(X^{\ee}\in
B_{\text{Sko}}(\varphi,\delta))\\\leq\mathbb{P}_{i}\left(\bigcap_{\ell=0}^{N_{\varphi}-1}\left\{T_{\ell}\in[t^\varphi_{\ell+1}-t^\varphi_{\ell}-2\delta,t^\varphi_{\ell+1}-t^\varphi_{\ell}+2\delta],\
    S_{\ell}=\varphi_{t^{\varphi}_{\ell}}
  \right\} \bigcap \left\{T_{N_{\varphi}}\geq T-t_{N_{\varphi}}-\delta,\
    S_{N_\varphi}=\varphi_{t^{\varphi}_{N_\varphi}} \right\}\right).
\end{multline*}
Hence, 
\begin{align*}
  \mathbb{P}_{i}(X^{\ee}\in B_{\text{Sko}}(\varphi,\delta))\leq&
  \left(\prod_{\ell=0}^{N_{\varphi}-1} \int_{(t^\varphi_{\ell+1}-t^\varphi_{\ell}-2\delta)\vee 0}^{t^\varphi_{\ell+1}-t^\varphi_{\ell}+2\delta} 
    \exp\left(-\frac{\mathfrak{T}\left(\varphi_{t^\varphi_{\ell}},\varphi_{t^\varphi_{\ell+1}}\right)}{\ee}\right) 
    \exp\left(-c^{\ee}(\varphi_{t^\varphi_{\ell}}) s_{\ell}\right)\ ds_{\ell} \right)
  \\&\quad\times\int_{(T-t^\varphi_{N_\varphi}-\delta)\vee 0}^{\infty}c^{\ee}\left(\varphi_{t^\varphi_{N_\varphi}}\right)\exp\left(-c^{\ee}\left(\varphi_{t^\varphi_{N_\varphi}}\right)s\right)ds
  \\=&\prod_{\ell=0}^{N_{\varphi}-1}\exp\left(-\frac{\mathfrak{T}\left(\varphi_{t^\varphi_{\ell}},\varphi_{t^\varphi_{\ell+1}}\right)}{\ee}\right)e^{-c^{\ee}\left(\varphi_{t^\varphi_{\ell}}\right)\left(t^\varphi_{\ell+1}-t^\varphi_{\ell}
    -2\delta\right)\vee 0}\ \frac{\left(1-e^{-4\delta c^{\ee}\left(\varphi_{t^\varphi_{\ell}}\right)} \right)}{c^{\ee}\left(\varphi_{t^\varphi_{\ell}}\right)}\\
&\quad \times e^{-c^{\ee}\left(\varphi_{t^\varphi_{N_{\varphi}}}\right)\left(T-t^\varphi_{N_{\varphi}}-\delta\right)\vee 0}.
\end{align*}
Now, using the facts that $c^{\ee}(i)\leq e^{-\frac{c}{\ee}}$ with $c>0$ and $(1-\exp(-4\delta
c^\varepsilon(\varphi_{t_\ell^\varphi})))/c^\varepsilon(\varphi_{t_\ell^\varphi})\rightarrow 4\delta$ when $\varepsilon\rightarrow
0$, we get 
\begin{align*}
\limsup_{\ee\rightarrow 0}\ee\log\mathbb{P}_{i}(X^{\ee}\in B_{\text{Sko}}(\varphi,\delta))\leq-\sum_{\ell=0}^{N_{\varphi}-1}\mathfrak{T}\left(\varphi_{t^\varphi_{\ell}},\varphi_{t^\varphi_{\ell+1}}\right).
\end{align*}
Similarly, for $\delta>0$ small enough, using the bound
\begin{equation*}
  \mathbb{P}_{i}(X^{\ee}\in
  B_{\text{Sko}}(\varphi,\delta)) \geq
  \mathbb{P}_{i}\left(\bigcap_{\ell=0}^{N_{\varphi}}\left\{T_{\ell}\in\left(t^\varphi_{\ell+1}-t^\varphi_{\ell}-\frac{\delta}{N_\varphi}\,
        ,\,t^\varphi_{\ell+1}-t^\varphi_{\ell}\right),\ S_{\ell}=\varphi_{t^{\varphi}_{\ell}}
    \right\} 
  \right),
\end{equation*}
we obtain
\begin{align*}
\liminf_{\ee\rightarrow 0}\ee\log\mathbb{P}_{i}(X^{\ee}\in B_{\text{Sko}}(\varphi,\delta))\geq-\sum_{\ell=0}^{N_{\varphi}-1}\mathfrak{T}\left(\varphi_{t^\varphi_{\ell}},\varphi_{t^\varphi_{\ell+1}}\right).
\end{align*}
Hence 
\[
\lim_{\varepsilon\rightarrow 0}\ee\log\mathbb{P}_{i}(X^{\ee}\in
B_{\text{Sko}}(\varphi,\delta))=-\sum_{\ell=0}^{N_{\varphi}-1}\mathfrak{T}\left(\varphi_{t^\varphi_{\ell}},\varphi_{t^\varphi_{\ell+1}}\right).
\]
This classically entails (see e.g.~\cite[Thm.\,4.1.11]{dembo-zeitouni}) that $X^\varepsilon$ satisfies a weak large deviations
principle with rate function $I_t$.
\end{proof}
Usually, a full large deviations principle is deduced from a weak one using exponential tightness of the laws of $X^\varepsilon$.
However, in our case, exponential tightness does not hold. This is due to the fact that the function $I_T$ is not a good rate
function, as can be seen from the following example: let $i$ and $j$ be two elements of $E$ and $s$ a real number in $(0,T)$. Now,
define for any positive integer $n$ large enough,
\[
\varphi_{n}(u)=
\left\{
  \begin{array}{ll}
    i& \text{ if } u\in[0,s)\\
    j& \text{ if } u\in[s,s+\frac{1}{n})\\
    i& \text{ if } u\in[s+\frac{1}{n},T].\\
  \end{array}
\right.
\]
Then, the subset $\left\{\varphi_{n}\mid n\in\mathbb{N}\backslash\{0\} \right\}$ is clearly non compact in
$\mathbb{D}([0,T],E)$ while $I_T$ is bounded on this set. 

To prove the full large deviations principle, we need the following lemma.

\begin{lem}
  \label{lem:sauts-bornes}
  For all $N\geq 1$ and $t>0$, we denote by $N^\varepsilon_t$ the number of jumps of $X^\varepsilon$ before $t$. There exists
  a constant $C_{N,t}\geq 0$ such that, for all $i\in E$,
  $$
  \limsup_{\varepsilon\rightarrow 0}\varepsilon\log\mathbb{P}_i(N^\varepsilon_t\geq N)\leq-C_{N,t},
  $$
  and for all $t>0$,
  $$
  \lim_{N\rightarrow\infty}C_{N,t}=+\infty.
  $$
\end{lem}

\begin{proof}
  Let us fix $x_0,x_1,\ldots,x_N\in E$ such that $\mathfrak{T}(x_i,x_{i+1})>0$ for all $i$. We compute
  \begin{align*}
    & \mathbb{P}_{x_0}(N^\varepsilon_t\geq N,\ X^\varepsilon_{J_i}=x_i,\ \forall i\in\{0,\ldots, N\}) \\ & 
    =\int_0^{t}\exp\left(-\frac{\mathfrak{T}(x_0,x_1)}{\ee}\right) 
    \exp(-c^{\ee}(x_0) s_{0})\ ds_0
    \ldots \\ & \qquad\qquad\ldots\int_{0}^{t-s_0-\ldots-s_{N-2}}\exp\left(-\frac{\mathfrak{T}(x_{N-1},x_N)}{\ee}\right) \exp(-c^{\ee}(x_{N-1}) s_{N-1})\
    ds_{N-1} \\
    & \leq \exp\left(-\frac{\mathfrak{T}(x_0,x_1)+\ldots+\mathfrak{T}(x_{N-1},x_N)}{\ee}\right) \int_0^{t}
    ds_0\ldots\int_{0}^{t-s_0-\ldots-s_{N-2}} ds_{N-1} \\
    & =\frac{t^N}{N!}\ \exp\left(-\frac{\mathfrak{T}(x_0,x_1)+\ldots+\mathfrak{T}(x_{N-1},x_N)}{\ee}\right).
  \end{align*}
  Therefore,
  \begin{multline*}
    \limsup_{\varepsilon\rightarrow 0}\varepsilon\log\mathbb{P}_{x_0}( N^\varepsilon_t\geq N,\ X^\varepsilon_{J_i}=x_i,\ \forall i\in\{0,\ldots, N\})
    \\ \leq -\mathfrak{T}(x_0,x_1)-\ldots-\mathfrak{T}(x_{N-1},x_N)\leq - N\inf_{i,j\in E}\mathfrak{T}(i,j)
  \end{multline*}
  Since the number of choices of $x_0,x_1,\ldots,x_N\in E$ is finite, we have proved Lemma~\ref{lem:sauts-bornes}.
\end{proof}

%

We can now prove that $X^{\ee}$ satisfies (a strong version of) the full LDP.
\begin{thm}
	\label{thm:upLDP}
For any measurable set $F$ of $\mathbb{D}([0,t],E)$,
\[
\limsup_{\ee\to 0}\ee\log\mathbb{P}_{i}(X^{\ee}\in F)\leq -\inf_{\varphi\in F}I_{t}(\varphi)\leq -\inf_{\varphi\in \bar{F}}I_{t}(\varphi).
\]
In particular, $X^{\ee}$ satisfies the large deviation principle with rate $\varepsilon^{-1}$ and rate function $I_{t}$.
\end{thm}
\begin{proof}
 Let $F\subset\mathbb{D}([0,t],E)$ be measurable. Set, for any positive integer $n$,
\[
F^{+}_{n}=\{\varphi\in F \mid N_{\varphi}\geq n \} \text{ and }F^{-}_{n}=F\backslash F^{+}_{n}.
\]
We have
\[
\mathbb{P}_{i}\left(X^{\ee}\in F^{-}_{n} \right)=\sum_{\ell=0}^{n-1}\sum_{(x_{1},\dots,x_{\ell})\in
  E^{\ell}}\mathbb{P}_{i}\left(\left\{X^{\ee}_{J_{k}}=x_{k},\ \forall k\leq N_{t}^{\ee},\ N_{t}^{\ee}=\ell\right\}\cap F \right).
\]
According to the computations made in the proof of Lemma \ref{lem:sauts-bornes}, we have
\[
\limsup_{\ee\to 0}\ee\log\mathbb{P}_{i}\left(X^{\ee}_{J_{k}}=x_{k},\ \forall k\leq N_{t}^{\ee},\ N_{t}^{\ee}=\ell \right)\leq -\left(\mathfrak{T}(i,x_{1})+\sum_{k=1}^{\ell-1}\mathfrak{T}(x_{k},x_{k+1})\right),
\]
which leads to
\[
\limsup_{\ee\to 0}\ee\log \mathbb{P}_{i}\left(X^{\ee}\in F^{-}_{n} \right) \leq \max\left\{ -\left(\mathfrak{T}(i,x_{1})+\sum_{k=1}^{\ell-1}\mathfrak{T}(x_{k},x_{k+1})\right) \right\},
\]
where the maximum is taken with respect to $\ell\in\{0,\ldots,n-1\}$ and to the elements $(x_{1},\dots,x_{\ell})$ of
$\cup_{\ell=0}^{n}E^{\ell}$ such that
\[
\left\{X^{\ee}_{J_{k}}=x_{k},\ \forall k\leq N_{t}^{\ee},\ N_{t}^{\ee}=\ell\right\}\cap F\neq\emptyset.
\]
This implies that
\[
\limsup_{\ee\to 0}\ee\log \mathbb{P}_{i}\left(X^{\ee}\in F^{-}_{n} \right) \leq -\inf_{\varphi\in F } I_{t}(\varphi).
\]
Finally, using Lemma \ref{lem:sauts-bornes}, we obtain
\[
\limsup_{\ee\to 0}\ee\log \mathbb{P}_{i}\left(X^{\ee}\in F \right)\leq \max\left(-\inf_{\varphi\in F} I_{t}(\varphi),-C_{n,t} \right).
\]
The result is now obtained by sending $n$ to infinity.
\end{proof}
The next theorem corresponds to Theorem \ref{thm:varad} in the discrete case situation. Unfortunately, as seen above, $I_T$ is not a
good rate function. This prevents us from applying directly Varadhan's lemma and leads to substantial difficulties. This is the place
where we need the assumption~(\ref{hyp:jumpMet}) on $\mathfrak{T}$. This result makes use of the sequence $(\varepsilon_k)_{k\geq 1}$
constructed in Lemma~\ref{lem:unifContinuity}, which holds true without modification in our discrete case. To avoid heavy notations,
we shall write ${\cal M}_s(i)$ for ${\cal M}_s(\{i\})$, where ${\cal M}_s$ is the measure constructed in
Lemma~\ref{lem:unifContinuity}.

\begin{thm}
	\label{thm:varad-discrete}
	For all $(t,i)$ in $(0,+\infty)\times E$,
	\begin{align}
	V(t,i) & :=\lim_{k\rightarrow\infty}\ee_k\log u^{\ee_k}(t,i) \notag \\ & =\sup_{\varphi\in\mathbb{D}([0,t],E)\text{ s.t.\
          }\varphi_0=i}\left\{-h(\varphi_t)+\int_0^t\sum_{j\in E}R(\varphi_s,j){\cal M}_{t-s}(j)ds-\sum_{0<s\leq
            t}\mathfrak{T}\left(\varphi_{s-},\varphi_s\right)\right\}. \label{eq:V-discrete}
	\end{align}
\end{thm}
\begin{proof}
  Let us fix a positive time $t$. To avoid heavy notations, we shall write $\varepsilon$ instead of $\varepsilon_{k}$ in this proof. Since the only part of the proof
  of Varadhan's Lemma relying on the compactness of the level sets of the rate function is the upper bound, we restrict ourself to
  this bound. As in the proof of Theorem \ref{thm:varad}, let $a\in\RR$ be smaller than the right-hand side of~\eqref{eq:V-discrete}
  and $C$ satisfying $|\Phi(\varphi)-h(\varphi(t))|\leq C(t+1)$.
  We define $K$ as the (non-necessarily compact) level set
  \[
  K=\{\varphi\in\mathbb{D}([0,t],E)\mid\varphi_0=i,\ I_t(\varphi)\leq C(t+1)-a \}.
  \]
  First, as in the proof of Theorem \ref{thm:varad}, we deduce from the LDP that
  \begin{multline}
    \label{eq:estim1}
    \limsup_{\ee\to0}\ee\log\mathbb{E}\left[\exp\left(\frac{1}{\ee}\left(-h_\ee(X^{\ee}_{t})+\int_{0}^{t}R(X^{\ee}_{s},v^{\ee}_{t-s})ds
        \right)\right) \mathds{1}_{X^{\ee}\in K^{c}}\right] \\\leq C(t+1)+\limsup_{\ee\to0}\ee\log\mathbb{P}\left(X^{\ee}\in K^{c} \right)
    \leq C(t+1)-\inf_{x\in K^{c}}I(x)\leq a.
  \end{multline}
  Second, by the definition of $I_t(\varphi)$, we have
  \begin{equation}
    \label{eq:nJumps}
    N_{\varphi}\leq \frac{I_t(\varphi)}{\min_{i,j\in E}\mathfrak{T}(i,j)}.
  \end{equation}
  Hence, according to \eqref{eq:nJumps}, there exists $N$ such that for all $\varphi$ in $K$, $N_{\varphi}\leq N$. 
  Fix $\gamma>0$. We deduce that $K=K_{N,\delta}\cup L_{N,\delta}$, where 
  \[
  K_{N,\gamma}=\{\varphi\in K\mid N_{\varphi}\leq N \text{ and } \inf_{i}(t^{\varphi}_{i+1}-t^{\varphi}_{i})\geq\gamma \}
  \]
  and
  \[
  L_{N,\gamma}=\{\varphi\in K\mid N_{\varphi}\leq N \text{ and } \inf_{i}(t^{\varphi}_{i+1}-t^{\varphi}_{i})<\gamma \}.
  \]
  Let $\Delta_{\beta}(T)$ be the set of subdivisions of $[0,t]$ with mesh greater that $\beta$. Since, for all $\varphi\in K_{N,\gamma}$ and all $\beta<\gamma$, we have
  \[
  \inf_{s\in\Delta_{\beta}(T)}\max_{1\leq i\leq |s|}\sup_{u,v\in[s_{i},s_{i+1})}|\varphi_{u}-\varphi_{v}|=0,
  \]
  we deduce from Arzela-Ascoli's theorem for the Skorokhod space that $K_{N,\gamma}$ is compact. 

  Fix $\delta>0$. We deduce that there exist $n\geq 1$ and $\varphi_{1}^{\gamma},\ldots,\varphi_{n}^{\gamma}\in K_{N,\gamma}$ such that
  \[
  K_{N,\gamma}=\bigcup_{i=1}^{n}G_{\varphi^{\gamma}_{i}},
  \]
  where the neighborhood $G_{\varphi^{\gamma}_{i}}$ of $\varphi_{i}^{\gamma}$ is chosen such that, for $\ee$ small enough,
  \begin{equation}
    \label{eq:ineq1}
    -h(\varphi^{\gamma}_{i}(t))+\Phi_{\ee}(\varphi^{\gamma}_{i})+\delta\geq \sup_{\varphi\in G_{\varphi^{\gamma}_{i}}} -h_{\varepsilon}(\varphi^{\gamma}(t))+\Phi_{\ee}(\varphi).
  \end{equation}
  Because of Lemma \ref{lem:sko}, we can also assume without loss of generality that $I_{t}$ is constant on
  $G_{\varphi_{i}^{\gamma}}$ for all $i$. Moreover, for all $i$ in $\{1,\dots,n\}$ and $\varepsilon$ small enough, we have
  \begin{equation}
    \label{eq:ineq2}
    \Phi(\varphi_{i}^{\gamma})-\delta\leq \Phi_{\ee}(\varphi_{i}^{\gamma})\leq \Phi(\varphi_{i}^{\gamma})+\delta.
  \end{equation}
  Following the lines of Theorem \ref{thm:varad}, we obtain
  \begin{align}
    \limsup_{\ee\to 0}\ee\log
    \mathbb{E} & \left[\exp\left(\frac{1}{\ee}\left(-h_{\ee}(X^{\ee}_{t})+\int_{0}^{t}R(X^{\ee}_{s},v^{\ee}_{t-s})ds \right)\right)
      \mathbbm{1}_{X^{\ee}\in K_{N,\gamma}}\ \right] \notag \\
    & \leq\max_{1\leq i \leq n}\left\{-h(\varphi^{\gamma}_{i}(t))+\int_0^t\sum_{j\in E}R(\varphi^\gamma_i(s),j){\cal
        M}_{t-s}(j)ds+2\delta-I_t(\varphi^{\gamma}_{i}) \right\} \notag \\
    &\leq\max_{\varphi\in K_{N,\gamma}}\left\{-h(\varphi(t))+\int_0^t\sum_{j\in E}R(\varphi(s),j){\cal
        M}_{t-s}(j)ds-I_t(\varphi) \right\}+2\delta. \label{eq:estim3}
  \end{align}

  We now prove a similar inequality when $X^{\ee}$ lies in $L_{N,\gamma}$. We first introduce some notations : for any
  $\varphi\in\mathbb{D}([0,T],E)$, in the case where $t^\varphi_{N_\varphi}\leq t-\gamma$, we define
  \[
  \Gamma^{\gamma}\varphi(s)=
  \left\{
    \begin{array}{ll}
      \varphi(s)& \text{ if } s\in[t^{\varphi}_{i},t^{\varphi}_{i+1})\text{ for }0\leq i\leq N_\varphi\text{ and } t^{\varphi}_{i+1}-t^{\varphi}_{i}\geq\gamma\\
      \varphi\left(t^{\varphi}_{\inf\{i\leq j\leq N_\varphi\mid t^{\varphi}_{j+1} -t^{\varphi}_{j}\geq \gamma \}}\right)& \text{ if }
      s\in[t^{\varphi}_{i},t^{\varphi}_{i+1})\text{ for }0\leq i\leq N_\varphi\text{ and } t^{\varphi}_{i+1}-t^{\varphi}_{i}<\gamma,\\
    \end{array}
  \right.
  \]
  with the convention that $t^\varphi_0=0$ and $t^\varphi_{N_\varphi+1}=t$.
  In the case where $t-\gamma<t^\varphi_{N_\varphi}\leq t$, we set $\Gamma^{\gamma}\varphi:=\Gamma^{\gamma}\widetilde{\varphi}$ where
  \[
  \widetilde{\varphi}:=
  \left\{
    \begin{array}{ll}
      \varphi(t)& \text{ if } s\in[t-\gamma,t],\\
      \varphi(s)& \text{ if } s<t-\gamma.
    \end{array}
  \right.
  \]
  Now, for all $\varphi$ in $L_{N,\gamma}$, $\Gamma^{\gamma}\varphi$ lies in $K_{N,\gamma}$. Indeed, by construction, $\Gamma^{\gamma}\varphi$ has inter-jumps times larger than $\gamma$ and, because of Hypothesis \eqref{eq:hypTij}, $I_{t}(\Gamma^{\gamma}\varphi)\leq I_{t}(\varphi)$ since $\Gamma^{\gamma}\varphi$ is obtained by suppressing jumps in $\varphi$. Hence, we get
  \begin{multline*}
    \mathbb{E}\left[\exp\left(\frac{1}{\ee}\left(-h_{\ee}(X^{\ee}_{t})+\int_{0}^{t}R(X^{\ee}_{s},v^{\ee}_{t-s})ds \right)\right) \mathbbm{1}_{X^{\ee}\in L_{N,\gamma}}\ \right]\\
    \leq\sum_{i=1}^{n}\mathbb{E}\left[\exp\left(\frac{1}{\ee}\left(-h_{\ee}(X^{\ee}_{t})+\int_{0}^{t}R(X^{\ee}_{s},v^{\ee}_{t-s})ds \right)\right) \mathbbm{1}_{X^{\ee}\in L_{N,\gamma}}\mathbbm{1}_{\Gamma^{\gamma}X^{\ee}\in G_{\varphi^{\gamma}_{i}}}\ \right].
  \end{multline*}
  For all $\varphi$ in $L_{N,\gamma}$, we necessarily have $\Gamma^{\gamma}\varphi(t)=\varphi(t)$, hence, it follows from the definition of $\Gamma^{\gamma}$ that
  \begin{align*}
    &-h_{\ee}(\varphi(t))+\int_{0}^{t}R(\varphi(s),v^{\ee}_{t-s})ds \\&=-h_{\ee}(\Gamma^{\gamma}\varphi(t))+\int_{0}^{t}R(\Gamma^{\gamma}\varphi(s),v^{\ee}_{t-s})ds+\sum_{t^{\varphi}_{j+1}-t^{\varphi}_{j}<\gamma}\int_{t^{\varphi}_{j}}^{t^{\varphi}_{j+1}}\left(R(\varphi(s),v^{\ee}_{t-s})- R(\Gamma^{\gamma}\varphi(s),v^{\ee}_{t-s})\right)ds\\
    &\leq-h_{\ee}(\Gamma^{\gamma}\varphi(t))+\int_{0}^{t}R(\Gamma^{\gamma}\varphi(s),v^{\ee}_{t-s})ds+2N\gamma \|R\|_{\infty}.
  \end{align*}
  Using this last inequality, \eqref{eq:ineq1} and \eqref{eq:ineq2}, we get
  \begin{align*}
    \mathbb{E}&\left[\exp\left(\frac{1}{\ee}\left(-h_{\ee}(X^{\ee}_{t})+\int_{0}^{t}R(X^{\ee}_{s},v^{\ee}_{t-s})ds \right)\right) \mathbbm{1}_{x\in L_{N,\gamma}}\ \right]\\
    &\leq\sum_{i=1}^{n}\mathbb{E}\left[\exp\left(\frac{1}{\ee}\left(-h_{\ee}(\Gamma^{\gamma}X^{\ee}_{t})+\int_{0}^{t}R(\Gamma^{\gamma}X^{\ee}_{0},v^{\ee}_{t-s})ds \right)\right)\exp\left(\frac{2}{\ee}N\gamma \|R\|_{\infty}\right) \mathbbm{1}_{x\in L_{N,\gamma}}\mathbbm{1}_{\Gamma^{\gamma}X^{\ee}\in G_{\varphi^{\gamma}_{i}}}\ \right]\\
    &\leq\sum_{i=1}^{n}\exp\left(\frac{1}{\ee}\left(-h(\varphi^{\gamma}_{i}(t))+\int_0^t\sum_{j\in
          E}R(\varphi^{\gamma}_{i}(s),j){\cal M}_{t-s}(j)ds+2\delta \right)\right) \\ & \qquad\qquad \qquad\qquad\exp\left(\frac{2}{\ee}N\gamma \|R\|_{\infty}\right) \mathbb{P}\left(X^{\ee}\in L_{N,\gamma},\Gamma^{\delta}X^{\ee}\in G_{\varphi^{\gamma}_{i}}  \right).
  \end{align*}
  It follows from Theorem \ref{thm:upLDP} that
  \begin{multline*}
    \limsup_{\ee\to 0}\ee\log\mathbb{E}\left[\exp\left(\frac{1}{\ee}\left(-h_{\ee}(X^{\ee}_{t})+\int_{0}^{t}R(X^{\ee}_{s},v^{\ee}_{t-s})ds \right)\right) \mathbbm{1}_{X^{\ee}\in L_{N,\gamma}}\right]\\
    \leq\max_{1\leq i \leq n}\left(-h(\varphi^{\gamma}_{i}(t))+\int_0^t\sum_{j\in E}R(\varphi^{\gamma}_{i}(s),j){\cal M}_{t-s}(j)ds+2\delta+2N\gamma\|R\|_{\infty}-\inf_{\varphi\in A_{\varphi_{i}^{\gamma}}}I_t(\varphi) \right)
  \end{multline*}
  (with the convention $\inf_{\varphi\in\emptyset} I_t(\varphi)=+\infty$), with
  \[
  A_{\varphi^{\gamma}_{i}}=\{\varphi\in L_{N,\gamma}\mid \Gamma^{\gamma}\varphi\in G_{\varphi^{\gamma}_{i}} \}.
  \]
  Now, using Lemma \ref{lem:sko} and $I_{t}(\Gamma^{\gamma}\varphi)\leq I_{t}(\varphi)$, we have for all $i$ such that
  $A_{\varphi^\gamma_i}\neq\emptyset$,
  \[
  I_t(\varphi_{i}^{\gamma})=\inf_{\varphi\in A_{\varphi_{i}^{\gamma}}} I_{t}(\Gamma^\gamma\varphi)\leq \inf_{\varphi\in A_{\varphi_{i}^{\gamma}}}I_{t}(x),
  \]
  which gives
  \begin{multline*}
    \limsup_{\ee\to
      0}\ee\log\mathbb{E}\left[\exp\left(\frac{1}{\ee}\left(-h_{\ee}(X^{\ee}_{t})+\int_{0}^{t}R(X^{\ee}_{s},v^{\ee}_{t-s})ds
        \right)\right) \mathbbm{1}_{X^{\ee}\in L_{N,\gamma}}\right]\\
    \begin{aligned}
      & \leq\max_{1\leq i \leq
        n}\left(-h_{\ee}(\varphi^{\gamma}_{i}(t))+\int_0^t\sum_{j\in E}R(\varphi^{\gamma}_{i}(s),j){\cal M}_{t-s}(j)ds
        +2\delta+2N\gamma\|R\|_\infty-I_t(\varphi^{\gamma}_{i}) \right) \\ &
      \leq\max_{\varphi\in K_{N,\gamma}}\left\{-h(\varphi(t))+\int_0^t\sum_{j\in E}R(\varphi(s),j){\cal
          M}_{t-s}(j)ds-I_t(\varphi) \right\}+2\delta+2N\gamma\|R\|_\infty.
    \end{aligned}
  \end{multline*}
  Combining the last inequality with \eqref{eq:estim1} and \eqref{eq:estim3}, we obtain
  \begin{multline*}
    \limsup_{\ee\to
      0}\ee\log\mathbb{E}\left[\exp\left(\frac{1}{\ee}\left(-h_{\ee}(X^{\ee}_{t})+\int_{0}^{t}R(X^{\ee}_{s},v^{\ee}_{t-s})ds
        \right)\right) \right]\\
    \begin{aligned}
      &\leq \max\left\{a; \max_{\varphi\in K_{N,\gamma}}\left[-h(\varphi(t))+\int_0^t\sum_{j\in E}R(\varphi(s),j){\cal
          M}_{t-s}(j)ds-I_t(\varphi) \right]+2\delta+2N\gamma\|R\|_{\infty}\right\} \\
 & \leq\sup_{\varphi\in\mathbb{D}([0,t],E)\text{ s.t.\
          }\varphi_0=i}\left\{-h(\varphi(t))+\int_0^t\sum_{j\in E}R(\varphi(s),j){\cal M}_{t-s}(j)ds-I_t(\varphi)\right\}+2\delta+2N\gamma\|R\|_{\infty}.
    \end{aligned}
  \end{multline*}
  Since $\delta$ and $\gamma$ are arbitrary, we have proved Theorem~\ref{thm:varad-discrete}.
\end{proof}

Our goal is now to obtain a version of Theorem \ref{thm:main-Lipschitz}. Since its proof makes use of the
translation invariance of Brownian motion, we cannot use the same method. In particular, we will see that the function
$t\mapsto\varepsilon\log u^\varepsilon(t,i)$ may not be uniformly Lipschitz for particular initial conditions. Hence we shall prove
directly the Lipschitz regularity of the limit $V$. For this, we first need the following lemma.

\begin{lem}
  \label{lem:V-Lipschitz}
  For all subsequence $(\varepsilon_k)_{k\geq 1}$ as in Theorem~\ref{thm:varad-discrete}, the limit $V(t,i)$ of $\varepsilon_{k}\log
  u^\varepsilon_{k}(t,i)$ satisfies, for all $t>0$ and all $i\neq j\in E$,
  \[
  V(t,i)\geq V(t,j)-\mathfrak{T}(i,j).
  \]
  In particular, this inequality is satisfied for all $t\geq 0$ if and only if $h(i)\leq h(j)+\mathfrak{T}(i,j)$ for all $i\neq j$.
\end{lem}

\begin{proof}
  Fix $i\neq j$, $t>0$ and $\eta>0$. Because of~(\ref{eq:V-discrete}), we can choose a function $\hat{\varphi}\in\mathbb{D}([0,t],E)$
  such that $\hat{\varphi}_0=j$ and
  \[
  V(t,j)\leq \eta-h(\hat{\varphi}_t)+\int_0^t\sum_{k\in E}R(\hat{\varphi}_s,k){\cal M}_{t-s}(k)ds-\sum_{0<s\leq
    t}\mathfrak{T}(\hat{\varphi}_{s-},\hat{\varphi}_s).
  \]
  For all $n\in\mathbb{N}$ such that $\frac{1}{n}<t^{\hat{\varphi}}_1\wedge t$, we define $\varphi^{(n)}\in\mathbb{D}([0,t],E)$ as
  \[
  \varphi^{(n)}_s=
  \begin{cases}
    i & \text{if }0\leq s<1/n \\
    \hat{\varphi}_s & \text{if }1/n\leq s\leq t.
  \end{cases}
  \]
  Using~(\ref{eq:V-discrete}) again, we have
  \begin{align*}
    \mathfrak{T}(i,j)+V(t,i) & \geq V(t,\varphi^{(n)}_t)+\int_0^t\sum_{k\in E}R(\varphi^{(n)}_s,k){\cal M}_{t-s}(j)ds-\sum_{1/n<s\leq
      t}\mathfrak{T}(\varphi^{(n)}_{s-},\varphi^{(n)}_s) \\ 
    & \geq V(t,\hat{\varphi}_t)-\frac{\|R\|_{{\cal H}}}{n}+\int_{1/n}^t\sum_{k\in E}R(\hat{\varphi}_s,k){\cal M}_{t-s}(j)ds-\sum_{0<s\leq
      t}\mathfrak{T}(\hat{\varphi}_{s-},\hat{\varphi}_s) \\
    & \geq V(t,j)-\eta-2 \frac{\|R\|_{{\cal H}}}{n},
  \end{align*}
  where $\|R\|_{{\cal H}}:=\sup_{i\in E,\ x\in{\cal H}}|R(i,x)|$. This concludes the proof letting $\eta\rightarrow 0$ and $n\rightarrow+\infty$.
\end{proof}

We can now state our result on the regularity of $V$.

\begin{thm}
  \label{thm:regularity-discrete}
  For all subsequence $(\varepsilon_k)_{k\geq 1}$ as in Theorem~\ref{thm:varad-discrete}, the limit $V(t,i)$ of $\varepsilon\log
  u^\varepsilon(t,i)$ is Lipschitz with respect to the time variable $t$ on $(0,+\infty)$. In addition, if $h(i)\leq h(j)+\mathfrak{T}(i,j)$ for all
  $i\neq j$, the function $V$ is Lipschitz on $\RR_+$.
\end{thm}

\begin{proof}
  Fix $t\geq 0$ and $i\in E$. For all $\delta>0$ and $\varepsilon>0$, proceeding as in the proof of Theorem~\ref{thm:main-Lipschitz},
  Markov's property entails
  \begin{align*}
    e^{-\frac{M\delta}{\varepsilon}}\EE_i[u^\varepsilon(t,X^\varepsilon_\delta)]\leq u^\varepsilon(t+\delta,i)\leq
    e^{\frac{M\delta}{\varepsilon}}\EE_i[u^\varepsilon(t,X^\varepsilon_\delta)].
  \end{align*}
  We can now estimate the distribution of $X^\varepsilon_\delta$ as follows. Let $N^\varepsilon$ be the number of jumps of
  $X^\varepsilon$ on the time interval $[0,\delta]$. For all $j\neq i$,
  \[
  \delta e^{-|E|\delta}e^{-\frac{\mathfrak{T}(i,j)}{\varepsilon}}\leq\PP_i(X^\varepsilon_{\delta}=j,\ N^\varepsilon\leq 1)=\int_0^\delta
  e^{-\frac{\mathfrak{T}(i,j)}{\varepsilon}}e^{-c^\varepsilon(i)s}e^{-c^\varepsilon(j)(\delta-s)}ds\leq \delta
  e^{-\frac{\mathfrak{T}(i,j)}{\varepsilon}}
  \]
  where $|E|$ is the cardinality of $E$. Similarly,
  \[
  1-|E|\delta\leq \PP_i(X^\varepsilon_\delta=i,\ N^\varepsilon=0)\leq 1
  \]
  and
  \[
  0\leq\PP_i(N^\varepsilon\leq 2)\leq \sum_{j\neq i}c_\varepsilon(j)
  e^{-\frac{\mathfrak{T}(i,j)}{\varepsilon}}\int_0^\delta(\delta-s)ds\leq \frac{|E|^2}{2}\delta^2.
  \]
  Putting all these inequalities together, we obtain
  \begin{align*}
    (1-|E|\delta)u^\varepsilon(t,i)+\sum_{j\neq i}\delta e^{-|E|\delta}e^{-\frac{\mathfrak{T}(i,j)}{\varepsilon}}u^\varepsilon(t,j) & \leq
    \EE_i[u^\varepsilon(t,X^\varepsilon_\delta)] \\ & \leq u^\varepsilon(t,i)+\sum_{j\neq i}\delta
    e^{-\frac{\mathfrak{T}(i,j)}{\varepsilon}} u^\varepsilon(t,j)+\delta^2\bar{C}\frac{|E|^2}{2},
  \end{align*}
  where $\bar{C}$ is the right-hand side of the main result in Lemma~\ref{lem:upBound}. Taking $\varepsilon\log$ of both sides of the
  last inequality and sending $\varepsilon$ to 0, we obtain
  \begin{equation}
  \label{eq:vtiBound}
  \max\left\{V(t,i);\max_{j\neq i} V(t,j)-\mathfrak{T}(i,j)\right\}-M\delta\leq V(t+\delta,i)\leq \max\left\{V(t,i);\max_{j\neq i}
    V(t,j)-\mathfrak{T}(i,j)\right\}+M\delta.
  \end{equation}
  Since $\delta>0$ was arbitrary, the result follows from Lemma~\ref{lem:V-Lipschitz}.
\end{proof}


\section{Study of the variational problem in the finite case}
\label{sec:uniqueness-E-finite}
Assume as above that $E$ is a finite set.
Our goal here is to study the limit problem 
\begin{equation}
\label{eq:VAR}
V(t,i)=\sup_{\varphi(0)=i}\left\{-h(\varphi(t))+\int_{0}^{t}\int_{\mathbb{R}^{r}} R(\varphi(u),y)\ \mathcal{M}_{t-u}(dy)\ du - I_{t}(\varphi) \right\}, \quad (t,i)\in\mathbb{R}_{+}\times E.
\end{equation}

A direct adaptation of Theorem~\ref{thm:varad-discrete} to obtain the dynamic programming version of~\eqref{eq:VAR}: for any $0\leq s\leq
t$, the limit of $\varepsilon_k\log u_{\varepsilon_k}$ satisfies
\[
V(t,i)=\sup_{\varphi(s)=i}\left\{V(s,\varphi(t))+\int_{s}^{t}\int_{
\mathbb{R}^{r}}R(\varphi(u),y)\ \mathcal{M}_{s+t-u}(dy)\ du-I_{s,t}(\varphi)  \right\},
\]
where, for all $\varphi$ in $\mathbb{D}([s,t],E)$
\[
I_{s,t}(\varphi)=\sum_{u\in(s,t]}\mathfrak{T}(\varphi(u-),\varphi(u)).
\]

In the sequel, we make use of the following assumptions on the dynamical systems related to problem \eqref{eq:EDO}.
For all $A\subset E$, we define the dynamical system in $\mathbb{R}_+^A:=\{(u_i,i\in A):u_i\geq 0 \text{ for all }i\in A\}$ denoted $S_{A}$ by
\begin{equation}
  \label{eq:syst-approche}
  \dot{u_{i}}=u_{i}\ R_{i}\left(\sum_{j\in A}\Psi_{l}(j)u_{j},\ 1\leq l\leq r\right),\ \quad i\in A, 
\end{equation}
where $R_{i}(x)$ stands for $R(i,x)$.

\medskip

\noindent
{\bf Hypothesis (H)}

\noindent
For all $A\subset E$, let $\text{Eq}_A$ be the set of steady states of $S_A$. Assume that all element of $\text{Eq}_A$ are hyperbolic
steady states and $S_{A}$ admits a unique steady state
\[
u^{\ast}_{A}=(u^{\ast}_{A,i})_{i\in A},
\]
such that, for all $i$ in $A$
\[
u^{\ast}_{A,i}=0\quad\Longrightarrow\quad R_{i}\left(\sum_{j\in E}\Psi_{l}(j)u^{\ast}_{A,j},\ 1\leq l\leq
  r\right)<0.
\]
Assume also that there exists a strict Lyapunov function $L_A:\mathbb{R}_+^A\rightarrow\RR$ for the dynamical system $S_A$, which
means that $L_A$ is $C^1$, admits as unique global minimizer on $\mathbb{R}_+^A$ and satisfies, for any solution $u(t)$ to $S_A$,
$$
\frac{d L_{A}(u(t))}{dt}=\sum_{i\in A}\frac{\partial L_A}{\partial u_i}(u(t))u_i(t) R_{i}\left(\sum_{j\in A}\Psi_{l}(j)u_{j}(A),\ 1\leq
  l\leq r\right)<0.
$$
for all $t\geq 0$ such that $u(t)\not\in \text{Eq}_A$.
\medskip

The hyperbolicity assumption means that for all steady state $u^{\ast}$, $u^{\ast}_{i}=0$ implies that 
$$
R_{i}\left(\sum_{j\in
    E}\Psi_{l}(j)u^{\ast}_{j},\ 1\leq l\leq r\right)\neq 0.
$$
For finite dimensional dynamical systems, it is well known that the hyperbolicity condition is generic under perturbation
(see~\cite{demazure}). Since hyperbolic equilibria are isolated, Hypothesis~(H) implies that $\text{Eq}_A$ is finite. The global
minimizer of $L_A$ is necessarily a stable steady state of $S_A$, hence it must be $u^\ast_A$ (since all the other steady states are
not stable). In addition, the equilibrium $u^\ast_A$ is globally asymptotically stable, in the sense that, for all initial condition
$u(0)=(u_i(0))_{i\in A}$ in $(0,+\infty)^A$, the solution to~\eqref{eq:syst-approche} converges to $u^\ast_A$. Indeed, since
$\dot{u}_i/u_i$ is uniformly bounded, $u(0)\in (0,+\infty)^A$ implies that $u(t)\in (0,+\infty)^A$ for all $t>0$. Now, the existence
of a strict Lyapunov function implies that $u(t)$ converges to an equilibrium $u^\ast$. If $u^{\ast}\neq u^\ast_A$, because of
Hypothesis (H), there exists $i$ such that $u^\ast_i=0$ and $\dot{u}_i(t)/u_i(t)\geq\frac{1}{2}R_{i}\left(\sum_{j\in
    E}\Psi_{l}(j)u^{\ast}_{j},\ 1\leq l\leq r\right)>0$ for all $t$ large enough. This is a contradiction with the convergence of
$u(t)$ to $u^\ast$.

General classes of dynamical systems satisfying Hypothesis~(H) have been given in~\cite{champagnat-jabin-raoul-10}. We give here two
examples.

\noindent
{\bf Example 1} 
Our first example corresponds to indirect competition for environmental resources and is an extension of the chemostat
model~\eqref{eq:chemostat}:
\[
R_{i}(v)=-d_{i}+c_{i}\sum_{l=1}^{r}\frac{\alpha_{l}\, \Psi_{l}(i)}{1+v_{l}},
\]
where $d_{i},c_{i},\alpha_{l}$ are positive real numbers satisfying $c_{i}\sum_{l=1}^{r}\alpha_{l}\Psi_{l}(i)>d_{i}$, for all $i$ in $E$.

\bigskip

\noindent 
{\bf Example 2} This example corresponds to direct competition of Lotka-Volterra (or logistic) type:
we assume that $E=\{1,\dots,r\}$ and
\[
R_{i}(v)=r_{i}-\sum_{l=1}^{r}v_{l},
\]
with the hypothesis that there exist positive constants $c_{1},\dots,c_{r}$ such that
\[
c_{i}\Psi_{i}(j)=c_{j} \Psi_{j}(i),\quad \forall i,j\in E,
\]
and
\[
\sum_{i,j\in E}x_{i}x_{j}c_i\Psi_{i}(j)>0,\quad \forall x\in \mathbb{R}^{r}\setminus{\{0\}}.
\]
In this example the matrix $(\Psi_{i}(j))_{i,j\in E}$ is interpreted as a competition matrix between the different types of
individuals. The last positivity assumption is satisfied for various competition
matrices~\cite{champagnat-jabin-raoul-10,champagnat-HDR}, for example, if
\[
c_{i}\Psi_{i}(j)=\beta_{i}\beta_{j}e^{\frac{|x_{i}-x_{j}|^{2}}{2}},
\] 
for any distinct $x_{1},\dots,x_{k}$ in $\mathbb{R}^{d}$ and  $\beta$ in $\mathbb{R}_{+}^{r}$. Another example is given by
\[
c_{i}\Psi_{i}(j)=\int_{\mathbb{R}^{d}}\ e^{(x_{i}+x_{j})\cdot z}\ \pi(dz),
\]
for all positive measure $\pi$ on $\mathbb{R}^{d}$ with support having non empty interior. 



Hypothesis~(H) implies a key property given in the next lemma.

\begin{lem}
  \label{lem:hittingTime}
  Assume Hypothesis~(H). For all $A\subset E$ and all $\rho>0$ small enough, the first hitting time $t^\ast_A(u(0),\rho)$ of the
  $\rho$-neighborhood of $u^\ast_A$ by a solution $u(t)$ to $S_A$ satisfies
  $$
  t_A^\ast(u(0),\rho)\leq C_\rho^\ast (1+\sup_{i\in A} -\log u_i(0))
  $$
  for some constant $C_\rho^\ast$ only depending on $\rho$.
\end{lem}

This lemma means that, when one coordinate $u_i(0)$ of $u(0)$ is close to zero, the time needed to converge to $u^\ast_A$
grows linearly with the logarithm of $(u_i(0))^{-1}$.

\begin{proof}
  Let $A$ be a non-empty subset of $E$ and $u$ be a solution of the dynamical system $S_{A}$. Without loss of generality, we can
  assume that $L_A(u^\ast_A)=0$, i.e.\ $\min_{u\in\RR_+^E}L_A(u)=0$. Let $\mathcal{H}_A$ be the set of $u\in\RR_+^A$ such that
  $\bar{u}\in\mathcal{H}$, where $\bar{u}\in\RR_+^E$ is obtained from $u$ by setting to zero the coordinates with indices in
  $E\setminus A$ and the set $\mathcal{H}$ was defined in Lemma \ref{lem:upBound}. Let
  $\mathcal{U}^{\ast}=(\text{Eq}_A\cap\mathcal{H}_A)\setminus\{u^{\ast}_{A}\}$. \medskip

  \noindent\emph{Step 1. Descrease of $L_A(u(t))$ between the visits of two neighborhoods of steady states}

  Let
  \[
  \underline{d}:=\min_{u^\ast,v^\ast\in\mathcal{U}^\ast \cup\{u^\ast_A\}}|u^\ast-v^\ast|
  \]
  and $\alpha_0>0$ such that 
  \[
  \sum_{i\in A}\partial_{i} L_A(x)x_{i} R_{i}\left(\sum_{j\in A}\Psi_{l}(j)x_{j},\ 1\leq
    l\leq r\right)<-\alpha_0, \quad \forall x\notin \bigcup_{u^{\ast}\in \mathcal{U}^{\ast}\cup\{u^\ast_A\}}B(u^{\ast},\underline{d}/4).
  \]
  Hence, setting $\|R\|_{\mathcal{H}}:=\sup_{i\in E,\ x\in \mathcal{H}}|R(i,x)|$, the decrease of $L_A(u(t))$ between two visits by
  $u(t)$ of two distinct balls $B(u^{\ast},\underline{d}/4)$ for $u^{\ast}\in \mathcal{U}^{\ast}\cup\{u^\ast_A\}$ is at least
  $\frac{\alpha_0\underline{d}}{2 \|R\|_{\mathcal{H}}}$. Now, let $\delta<\underline{d}/4$ be small enough to have, for all
  $u^\ast\in\mathcal{U}^\ast$,
  \[
  \sup_{u\in B(u^\ast,\delta)} L_A(u)-\inf_{u\in B(u^\ast,\delta)} L_A(u)<\frac{\alpha_0\underline{d}}{2 \|R\|_{\mathcal{H}}}.
  \]
  Hence, defining for all $k\in\NN$,
  \[
  \mathcal{H}_A^k:=\left\{u\in\mathcal{H}_A\mid k \frac{\alpha_0\underline{d}}{2 \|R\|_{\mathcal{H}}}\leq L_A(u)<(k+1)
    \frac{\alpha_0\underline{d}}{2 \|R\|_{\mathcal{H}}}\right\},
  \]
  the solution $u(t)$ can only visit at most one ball of type $B(u^\ast,\delta)$ for some $u^\ast\in\mathcal{U}^\ast$ during its
  travelling time through $\mathcal{H}_{A}^{k}$.
  \medskip

  \noindent\emph{Step 2. Time spent in $\mathcal{H}^k_A$}

  According to Hypothesis~(H), for any steady state $u^{\ast}\in \mathcal{U}^{\ast}$, we have
  \[
  R_{i}\left(\sum_{j\in
      A}\Psi_{l}(j)u^{\ast}_{j},\ 1\leq l\leq r\right)> 0, \quad \text{for at least one } i(u^\ast)\in A \text{ such that } u^{\ast}_{i(u^\ast)}= 0.
  \]
  Consequently, reducing $\delta>0$ if necessary, there exists a positive real number $R_{-}$ such that, for all $u^\ast\in
  \mathcal{U}^{\ast}$,
  \begin{equation}
    \label{eq:borne-R}
    \forall x\in B(u^{\ast},\delta),\quad R_{i(u^\ast)}\left(\sum_{j\in
        A}\Psi_{l}(j)u^{\ast}_{j},\ 1\leq l\leq r\right)> R_{-}.
  \end{equation}
  In addition, since $L_{A}$ is a strict Lyapunov function, there exists a positive real number $\alpha$ such that 
  \begin{equation}
    \label{eq:borne-L_A}
    \sum_{i\in A}\partial_{i} L_A(x)x_{i} R_{i}\left(\sum_{j\in A}\Psi_{l}(j)x_{j},\ 1\leq
      l\leq r\right)<-\alpha, \quad \forall x\notin \bigcup_{u^{\ast}\in \mathcal{U}^{\ast}\cup\{u^\ast_A\}}B(u^{\ast},\delta).
  \end{equation}
  Now, let $c$ be a positive real number satisfying
  \[
  c^{-1}<\alpha^{-1}\|R\|_{\mathcal{H}}.
  \]

  Let $k_{0}$ such that $u_{0}\in \mathcal{H}^{k_{0}}_{A}$. For any $0\leq k\leq k_{0}$, let $t_{k}$ be the hitting time of
  $\mathcal{H}^{k}$ and $t_{-1}$ the first hitting time of $B(u^\ast_A,\delta)$, and define for all $k\geq 0$
  \[
  F_{k}(t)=
  \begin{cases}
    L_{A}(u(t)) & \text{if }u(t)\notin \bigcup_{u^{\ast}\in \mathcal{U}^{\ast}\cup\{u^\ast_A\}}B(u^{\ast},\delta)\text{ for all
    }t\in[t_k,t_{k-1}], \\
    L_{A}(u(t))-c\log u_{i(u^{\ast})}(t) &\text{if }\exists t\in[t_k,t_{k+1}],\ \exists u^*\in\mathcal{U}^*\text{ such that }u(t)\in B(u^{\ast},\delta).
  \end{cases}
  \]
  Note that Step 1 implies that, for all $k\geq 0$, there exists at most one $u^*\in\mathcal{U}^*$ such that $u(t)\in
  B(u^{\ast},\delta)$ for some $t\in[t_k,t_{k+1}]$. Using~\eqref{eq:borne-L_A},~\eqref{eq:borne-R}, the fact that
  $\frac{d}{dt}L_A(u(t))\leq 0$ and the definition of $\|R\|_{\mathcal{H}}$, we have for all $t\in[t_k,t_{k+1}]$,
  \[
  \frac{dF_k(t)}{dt}\leq(-c R_-)\vee(-\alpha+c\|R\|_{\mathcal{H}})<0.
  \]
  Therefore, for all $t\in[t_k,t_{k+1}]$,
  \[
  F_k(t)\leq (k+1) \frac{\alpha_0\underline{d}}{2 \|R\|_{\mathcal{H}}}-c\inf_{i\in A}\log u_i(t_k).
  \]
  In addition it follows from Lemma~\ref{lem:upBound} that
  \[
  F_k(t)\geq k \frac{\alpha_0\underline{d}}{2
    \|R\|_{\mathcal{H}}}-\log\left(\frac{v_{\textnormal{max}}+A(|E|-1)e^{-\gamma/\varepsilon}}{\Psi_{\textnormal{min}}}\right).
  \]
  Setting $k=k_0$, we deduce that there exists a constant $C_{k_0-1}>0$ such that
  \[
  t_{k_0-1}\leq C_{k_0-1}(1-\inf_{i\in A}\log u_i(0))
  \]
  and, since $\left|\frac{\dot{u}_i}{u_i}\right|\leq\|R\|_\mathcal{H}$,
  \[
  -\inf_{i\in A}\log u_i(t_{k_0-1})\leq -\inf_{i\in A}\log u_i(0)+\|R\|_\mathcal{H}C_{k_0-1}(1-\inf_{i\in A}\log u_i(0)).
  \]
  Proceeding by induction, it follows that there exist constants $C_k$ and $D_k$ depending only on $k_0$ and $u(0)$ such that, for
  all $k\geq 1$,
  \[
  t_{k-1}-t_k\leq C_k(1-\inf_{i\in A}\log u_i(0))\quad\text{and}\quad -\inf_{i\in A}\log u_i(t_{k-1})\leq D_k(1-\inf_{i\in A}\log u_i(0)).
  \]
  Similarly, there exist constants $C_k$ and $D_k$ such that
  \[
  t_{-1}-t_0\leq C_0(1-\inf_{i\in A}\log u_i(0))\quad\text{and}\quad -\inf_{i\in A}\log u_i(t_{-1})\leq D_0(1-\inf_{i\in A}\log u_i(0)).
  \]
  Since $L$ is a strict Lyapunov function, for all $\rho<\delta$, the time needed to enter $B(u^\ast_A,\rho)$ starting from any point
  in $B(u^\ast_A,\delta)$ is bounded by a constant depending only on $\rho$. The result follows.
\end{proof}

This lemma entails the following property.
\begin{prop}
  \label{prop:hyp3}
  Assume Hypothesis~(H). Let $(\ee_{k})_{k\geq 1}$ be as in Theorem~\ref{thm:regularity-discrete}. For any
  $t\geq 0$, there exists $\rho_t>0$ such that, for all $s\in(t,t+\rho_t]$, $v^\varepsilon_{s}$ converges to $F(\{V(t,\cdot)=0\})$,
  where the convergence is uniform in all compact subsets of $(t,t+\rho_t]$ and where
  \[
  F(A)=\left(\sum_{j=1}^{r}\eta_{i}(j)u^{\ast}_{A,j}\right)_{1\leq i \leq r},\quad \forall A\subset E.
  \]
  In particular, the weak limit $\mathcal{M}_{s}$ of $\delta_{v_{\ee_{k}}(s)}$ obtained in Lemma~\ref{lem:unifContinuity} satisfies
  \[
  \mathcal{M}_{s}=\delta_{F\left(\{V(t,\cdot)=0\}\right)},\text{\ for almost all }s\in(t,t+\rho_t).
  \]
  and the function $t\mapsto F(\{V(t,\cdot)=0\})$ is right-continuous.
\end{prop}

Note that, because the set $E$ is finite, the range of the function $F$ is finite, and thus, assuming that $t\mapsto
F(\{V(t,\cdot)=0\})$ is right-continuous implies that the union of the time intervals where this function is constant is equal to
$\RR_+$. We emphasize that this is not true in general for measurable functions taking values in finite sets.

\begin{proof}
Let us fix $t\geq 0$ and define $A=\{i\in E \mid V(t,i)=0 \}$. Since, at time $t$, we have $V(t,i)=\lim_{k\to\infty}\ee_{k}\log
u^{\ee_{k}}(t,i)$, it follows that, for all $\delta>0$, for all $k$ large enough,
\begin{equation}
\label{eq:initEstimate}
u^{\ee_{k}}(t,i)\geq e^{-\frac{\delta}{\ee_{k}}},\quad \forall i\in  A.
\end{equation}
In addition, if $t>0$, using the continuity of $V$ at time $t$ (see Theorem~\ref{thm:regularity-discrete}), there exists some
positive real numbers $\eta_{t}$ and $\alpha$ such that $V(t+s,i)<-\alpha$ for all $i\in E\setminus A$ and all $s\in[0,\eta_{t}]$.
If $t=0$, setting
$$\beta=\max\left(\max_{i\in E\setminus A} -h(i),\ \max_{i\neq j}-h(j)-\mathfrak{T}(i,j)\right)<0$$
and $-\alpha\in (\beta,0)$, it follows from  the upper bound of \eqref{eq:vtiBound} that
$V(t+s,i)\leq-\alpha$ for all $i\in E\setminus A$ and all $s\in[0,\eta_{t}]$ where
\[
\eta_{t}=\frac{\beta-\alpha}{M}.
\]
In both cases, we obtain
\begin{equation}
\label{eq:initEstimate2}
\sup_{s\in[0,\eta_{t}]}u^{\ee_{k}}(t+s,i)\leq e^{-\frac{\alpha}{\ee_{k}}},\quad \forall i\in E\setminus A.
\end{equation}

Now, let $(u_{i}(s), s\in\mathbb{R}_{+})_{i\in A}$ be the solution of the dynamical system $S_{A}$, as defined in \eqref{eq:syst-approche}, with initial conditions $u_{i}(0)=u^{\ee_{k}}(t,i)$, for all $i$ in $A$. According to \eqref{eq:syst-approche} and \eqref{eq:EDO}, we have, for any time $s$ and any $i$ in $A$,
\begin{align*}
& \left|u^{\ee_{k}}(t+s,i)-u_{i}(s/\ee_{k})\right|\\  
&\qquad\qquad\leq\frac{1}{\ee_{k}}\int_{0}^{s}\left|u^{\ee_{k}}(t+u,i)R(i,v^{\ee_{k}}_{t+u})-u_{i}(u/\ee_{k})R_{i}\left(\sum_{j\in A}\Psi_{l}(j)u_{j}(u/\varepsilon_k),\ 1\leq l\leq
  k\right)\right|\, du\\
  &\qquad\qquad\quad+\int_{0}^{s}\sum_{j\in E}e^{-\frac{\mathfrak{T}(i,j)}{\ee_{k}}}\left|u^{\ee_{k}}(u,j)-u^{\ee_{k}}(u,i) \right|\, du.
\end{align*}
Using Hypothesis \eqref{it:1}, in conjunction with Lemma \ref{lem:upBound}, we get
\[
\left|u^{\ee_{k}}(t+s,i)-u_{i}(s/\ee_{k})\right|\leq \frac{D}{\ee_{k}}\int_{0}^{s}\left|u^{\ee_{k}}(t+u,i)-u_{i}(u/\ee_{k})\right|\,
du+2 Cse^{-\frac{\gamma}{\ee_{k}}},
\]
with $\gamma$ defined in Lemma \ref{lem:upBound} and some positive constants $C$ and $D$. Hence, Gronwall Lemma entails that
\begin{equation}
\label{eq:GronwalEstimate}
\left|u^{\ee_{k}}(t+s,i)-u_{i}(s/\ee_{k})\right|\leq \frac{\tilde{C}}{\ee_{k}}\exp\left(\frac{Ds-\gamma}{\ee_{k}} \right),
\end{equation}
for some positive constant $\tilde{C}$.
\medskip

Now, let us compare $u^{\ee_{k}}(t+s,i)$ with the expected limit $u^{\ast}_{i,A}$ which gives
\[
\left|u^{\ee_{k}}(t+s,i)-u^{\ast}_{i,A}\right|\leq \left|u^{\ee_{k}}(t+s,i)-u_{i}(s/\ee_{k})\right|+\left|u_{i}(s/\ee_{k})-u^{\ast}_{i,A}\right|.
\]
According to Lemma \ref{lem:hittingTime}, we have, for any postive real number $\rho$,
\[
\frac{s}{\ee_{k}}>C_\rho^\ast (1+\sup_{i\in A} -\log u_i(0))\Longrightarrow \left|u_{i}(s/\ee_{k})-u^{\ast}_{i,A}\right|<\rho.
\]
However, according to \eqref{eq:initEstimate}, for all $\delta>0$, for $k$ large enough,
\[
-\log u_i(0)\leq \frac{\delta}{\ee_{k}},\quad \forall i\in A..
\]
This last inequality gives, for all $k$ large enough, 
\[
s>2\delta C_\rho^\ast\Longrightarrow \left|u_{i}(s/\ee_{k})-u^{\ast}_{i,A}\right|<\rho.
\]
This entails in conjonction with \eqref{eq:GronwalEstimate} that 
\[
\limsup_{k\rightarrow+\infty}\sup_{s\in [2\delta C^*_\rho, \gamma/2D]}\left|u^{\ee_{k}}(t+s,i)-u^{\ast}_{i,A}\right|<\rho,\quad \forall i\in A.
\]
Since  $\rho$ and $\delta$ were arbitrary, the desired uniform convergence follows from~\eqref{eq:initEstimate2} with $\rho_{t}=\frac{\gamma}{2D}\wedge \eta_{t}$. The remaining statements follow easily.
\end{proof}

\begin{cor}
  \label{cor:discrete}
  Assume Hypothesis~(H). Any limit $V$ of $\varepsilon_k\log u^{\varepsilon_k}$ along a subsequence as in
  Theorem~\ref{thm:regularity-discrete} satisfies $V(0,i)=-h(i)$ for all $i\in E$ and for all $t\geq 0$
  \begin{equation*}
    V(t,i)=\sup_{\varphi(0)=i}\left\{-h(\varphi(t))+\int_{0}^{t} R(\varphi(u),F(\{V(t-u,\cdot)=0 \}))\ du - I_{t}(\varphi) \right\},
   \end{equation*}
  and its dynamic programming version
  \begin{equation}
    \label{eq:VARN}
    V(t,i)=\sup_{\varphi(s)=i}\left\{V(s,\varphi(t))+\int_{s}^{t} R(\varphi(u),F(\{V(t+s-u,\cdot)=0 \}))\  du-I_{s,t}(\varphi)  \right\}.
  \end{equation}
  In addition, the problem~\eqref{eq:VARN} admits a unique solution such that $t\mapsto F(\{V(t,\cdot)=0\})$ is right-continuous. In
  particular, the full sequence $(\varepsilon\log u^\varepsilon)_{\varepsilon>0}$ converges to this unique solution when
  $\varepsilon\rightarrow 0$.
\end{cor}


\begin{proof}
  The only non-obvious consequence of Proposition~\ref{prop:hyp3} is the uniqueness of a solution to~\eqref{eq:VARN} with $t\mapsto
  F(\{V(t,\cdot)=0\})$ right-continuous. To prove this, observe that, by continuity of $V$ for $t>0$ and using~\eqref{eq:vtiBound}
  for $t=0$, the set
  \[
  A=\{t\geq 0\mid\text{there exists a unique solution up to time }t\}
  \]
  cannot have a finite upper bound.
\end{proof}

We conclude by giving the discrete Hamilton-Jacobi formulation of the variational problem solved by $V$.

\begin{thm}
  \label{thm:HJD}
  Under Hypothesis~(H) and assuming that $h(i)\leq h(j)+\mathfrak{T}(i,j)$ for all $i\neq j$, the problem
  \begin{equation}
    \label{eq:HJD}
    \dot{V}(t,i)=\sup\left\{R_{j}(F(\{V(t,\cdot)=0\})) \mid j\in E, V(t,j)-\mathfrak{T}(j,i)=V(t,i)\right\},\quad V(0,i)=h(i)\
    \forall i\in E
  \end{equation}
  (with the convention $\mathfrak{T}(i,i)=0$) admits a unique solution such that $t\mapsto F(\{V(t,\cdot)=0\})$ is right-continuous
  which is the unique solution to the variational problem \eqref{eq:VARN}.
\end{thm}

The result also extends to cases where the condition $h(i)\leq h(j)+\mathfrak{T}(i,j)$ is not satisfied for some $i\neq j$. In this
case, one must replace the initial condition in~\eqref{eq:HJD} by 
\[
V(0,i)=\max\left\{-h(i);\max_{j\neq i}-h(j)-\mathfrak{T}(i,j)\right\},\quad\forall i\in E
\]
and the solution of~\eqref{eq:HJD} coincides with the solution to the variational problem~\eqref{eq:VARN} for positive times only.

\begin{proof}
  We first prove the uniqueness part and then that the solution of~\eqref{eq:VARN} also solves~\eqref{eq:HJD}.
  \medskip

  \noindent \emph{Step 1: Uniqueness for~\eqref{eq:HJD}.}
    
  Let $t\in[0,+\infty]$ be the largest time such that there is uniqueness for~\eqref{eq:HJD} up to time $t$ and assume $t<\infty$.
  Let $i\in E$ be such that $V(t,i)=0$. Since $\mathfrak{T}(j,i)>0$ for all $j\neq i$ and because of the Lipschitz regularity of any
  solution $V$ of~\eqref{eq:HJD}, $\dot{V}(s,i)=R_i( F(\{V(t,\cdot)=0\}))$ for all $s\in[ t,t+\delta_t]$ for some $\delta_t>0$. Hence
  $V(s,i)$ is uniquely determined for all $i$ such that $V(t,i)=0$ and $s\in[ t,t+\delta_t]$. We proceed similarly for all the $i'\in
  E$ such that $V(t,i')\in[-\inf_{i\neq j}\mathfrak{T}(i,j)/2,0)$: their dynamics is determined only by $ F(\{V(t,\cdot)=0\})$ and
  $V(s,i)$ for all $i$ such that $V(t,i)=0$, for $s\geq t$ in the time interval $[ t,t+\delta_t]$. We obtain a similar result
  inductively for all the $i'\in E$ such that $V(t,i')\in[-k\inf_{i\neq j}\mathfrak{T}(i,j)/2,-(k-1)\inf_{i\neq
    j}\mathfrak{T}(i,j)/2)$ for all $k\geq 1$. This contradicts the finiteness of $t$ and hence uniqueness for~\eqref{eq:HJD} is
  proved. 
  \medskip
 
  \noindent\emph{Step 2: The function $V$ of Corollary~\ref{cor:discrete} solves~\eqref{eq:HJD}.}
    
  Let $t\geq 0$ and $i\in E$ be fixed and let us prove that the solution $V$ of~\eqref{eq:VARN} is right-differentiable with respect
  to time at $(t,i)$ with derivative given by~\eqref{eq:HJD}. According to Corollary \ref{cor:discrete}, we have for all $\delta>0$
  \begin{equation}
    \label{eq:pf-HJD}
    V(t+\delta,i)=\sup_{\varphi(t)=i}\left\{V(t,\varphi(t+\delta))+\int_{t}^{t+\delta} R(\varphi(u),F(\{V(t+\delta-u,\cdot)=0 \}))\
      du-I_{t,t+\delta}(\varphi)  \right\}
  \end{equation}
  and there exists $\delta_{t}>0$, such that $F(\{V(t+\delta,\cdot)=0 \})=F(\{V(t,\cdot)=0 \})$ for all $\delta<\delta_{t}$. 

  We know from Lemma~\ref{lem:V-Lipschitz} that $V(t,i)\geq V(t,j)-\mathfrak{T}(i,j)$ for all $i\neq j$. If the inequality is strict
  for all $j\neq i$, it is clear that the supremum in~\eqref{eq:pf-HJD} is attained for the constant function $\varphi\equiv i$ for
  sufficiently small $\delta>0$ and hence $\dot{V}(t+\delta,i)=R_i(F(\{V(t,\cdot)=0\})$, which entails~\eqref{eq:HJD}.

  If there exists $j\neq i$ such that $V(t,i)= V(t,j)-\mathfrak{T}(i,j)$, let $j^*\in E$ be such that $R_{j}(F(\{V(t,\cdot)=0\})$ is
  maximal among all $j\in E$ such that $V(t,i)= V(t,j)-\mathfrak{T}(i,j)$. Since the supremum in~\eqref{eq:pf-HJD} cannot be attained
  for functions $\varphi$ jumping two times or more in $[t,t+\delta]$ for $\delta>0$ small enough, considering all possible choices
  for $\varphi$ with less than one jump, one easily checks that, for $\delta>0$ small enough,
  \begin{align*}
    V(t+\delta,i) & =\lim_{n\rightarrow +\infty} V(t,\varphi^{(n)}(t+\delta))+\int_{t}^{t+\delta} R(\varphi^{(n)}(u),F(\{V(t,\cdot)=0
    \}))\ du-\mathfrak{T}(i,j^*) \\ & =V(t,j^*)-\mathfrak{T}(i,j^*)+\delta R_{j^*}(F(\{V(t,\cdot)=0 \}) \\ & =V(t,i)+\delta
    R_{j^*}(F(\{V(t,\cdot)=0 \})
  \end{align*}
  where $\varphi^{(n)}_{t+s}=i$ if $s< 1/n$ and $\varphi^{(n)}_{t+s}=j^*$ if $1/n\leq s\leq \delta$. Hence
  $\dot{V}(t+\delta,i)=R_{j^*}(F(\{V(t,\cdot)=0 \})$, which gives the result.
%
\end{proof}

\bigskip

\paragraph{Acknowledgments} Nicolas Champagnat benefited from the support of the Chair ``Mod\'elisation Math\'ematique et
Biodiversit\'e'' of Veolia Environnement - \'Ecole Polytechnique - Mus\'eum National d'Histoire Naturelle - Fondation X and from the
support of the French National Research Agency for the project ANR-14-CE25-0013, ``Propagation phenomena and nonlocal equations''
(ANR NONLOCAL).

\bibliographystyle{plain}
\bibliography{biblio}

\begin{thebibliography}{10}

\bibitem{bardi-capuzzo}
Martino Bardi and Italo Capuzzo-Dolcetta.
\newblock {\em Optimal control and viscosity solutions of
  {H}amilton-{J}acobi-{B}ellman equations}.
\newblock Systems \& Control: Foundations \& Applications. Birkh\"auser Boston,
  Inc., Boston, MA, 1997.
\newblock With appendices by Maurizio Falcone and Pierpaolo Soravia.

\bibitem{barles-evans-al-90}
G.~Barles, L.~C. Evans, and P.~E. Souganidis.
\newblock Wavefront propagation for reaction-diffusion systems of {PDE}.
\newblock {\em Duke Math. J.}, 61(3):835--858, 1990.

\bibitem{barles-mirrahimi-al-09}
Guy Barles, Sepideh Mirrahimi, and Beno{\^{\i}}t Perthame.
\newblock Concentration in {L}otka-{V}olterra parabolic or integral equations:
  a general convergence result.
\newblock {\em Methods Appl. Anal.}, 16(3):321--340, 2009.

\bibitem{barles-perthame-07}
Guy Barles and Beno{\^{\i}}t Perthame.
\newblock Concentrations and constrained {H}amilton-{J}acobi equations arising
  in adaptive dynamics.
\newblock In {\em Recent developments in nonlinear partial differential
  equations}, volume 439 of {\em Contemp. Math.}, pages 57--68. Amer. Math.
  Soc., Providence, RI, 2007.

\bibitem{Billingsley}
Patrick Billingsley.
\newblock {\em Weak convergence of measures: {A}pplications in probability}.
\newblock Society for Industrial and Applied Mathematics, Philadelphia, Pa.,
  1971.
\newblock Conference Board of the Mathematical Sciences Regional Conference
  Series in Applied Mathematics, No. 5.

\bibitem{bogachev}
V.~I. Bogachev.
\newblock {\em Measure theory. {V}ol. {I}, {II}}.
\newblock Springer-Verlag, Berlin, 2007.

\bibitem{champagnat-HDR}
Nicolas Champagnat.
\newblock {\em {Stochastic and deterministic approaches in Biology: adaptive
  dynamics, ecological modeling, population genetics and molecular dynamics;
  Well-posedness for ordinary and stochastic differential equations}}.
\newblock Habilitation {\`a} diriger des recherches, {Universit{\'e} de
  Lorraine}, February 2015.

\bibitem{champagnat-jabin-11}
Nicolas Champagnat and Pierre-Emmanuel Jabin.
\newblock The evolutionary limit for models of populations interacting
  competitively via several resources.
\newblock {\em J. Differential Equations}, 251(1):176--195, 2011.

\bibitem{champagnat-jabin-meleard-13}
Nicolas Champagnat, Pierre-Emmanuel Jabin, and Sylvie M\'el\'eard.
\newblock Adaptation in a stochastic multi-resources chemostat model.
\newblock {\em J. Math. Pures Appl. (9)}, 101(6):755--788, 2014.

\bibitem{champagnat-jabin-raoul-10}
Nicolas Champagnat, Pierre-Emmanuel Jabin, and Ga\"el Raoul.
\newblock Convergence to equilibrium in competitive {L}otka-{V}olterra and
  chemostat systems.
\newblock {\em C. R. Math. Acad. Sci. Paris}, 348(23-24):1267--1272, 2010.

\bibitem{clarke}
Francis Clarke.
\newblock {\em Functional analysis, calculus of variations and optimal
  control}, volume 264 of {\em Graduate Texts in Mathematics}.
\newblock Springer, London, 2013.

\bibitem{demazure}
Michel Demazure.
\newblock {\em Bifurcations and catastrophes}.
\newblock Universitext. Springer-Verlag, Berlin, 2000.
\newblock Geometry of solutions to nonlinear problems, Translated from the 1989
  French original by David Chillingworth.

\bibitem{dembo-zeitouni}
Amir Dembo and Ofer Zeitouni.
\newblock {\em Large deviations techniques and applications}.
\newblock Jones and Bartlett Publishers, Boston, MA, 1993.

\bibitem{desvillettes-jabin-al-08}
Laurent Desvillettes, Pierre-Emmanuel Jabin, St{\'e}phane Mischler, and
  Ga{\"e}l Raoul.
\newblock On selection dynamics for continuous structured populations.
\newblock {\em Commun. Math. Sci.}, 6(3):729--747, 2008.

\bibitem{dieckmann-doebeli-99}
U.~Dieckmann and M.~Doebeli.
\newblock On the origin of species by sympatric speciation.
\newblock {\em Nature}, 400:354--357, 1999.

\bibitem{diekmann-jabin-al-05}
Odo Diekmann, Pierre-Emanuel Jabin, St\'{e}phane Mischler, and Beno\^{i}t
  Perthame.
\newblock The dynamics of adaptation: An illuminating example and a
  {H}amilton-{J}acobi approach.
\newblock {\em Theor. Pop. Biol.}, 67:257--271, 2005.

\bibitem{engel}
Klaus-Jochen Engel and Rainer Nagel.
\newblock {\em One-parameter semigroups for linear evolution equations}, volume
  194 of {\em Graduate Texts in Mathematics}.
\newblock Springer-Verlag, New York, 2000.
\newblock With contributions by S. Brendle, M. Campiti, T. Hahn, G. Metafune,
  G. Nickel, D. Pallara, C. Perazzoli, A. Rhandi, S. Romanelli and R.
  Schnaubelt.

\bibitem{LDPKurtz}
Jin Feng and Thomas~G. Kurtz.
\newblock {\em Large deviations for stochastic processes}, volume 131 of {\em
  Mathematical Surveys and Monographs}.
\newblock American Mathematical Society, Providence, RI, 2006.

\bibitem{fleming-soner-93}
Wendell~H. Fleming and H.~Mete Soner.
\newblock {\em Controlled {M}arkov processes and viscosity solutions},
  volume~25 of {\em Applications of Mathematics (New York)}.
\newblock Springer-Verlag, New York, 1993.

\bibitem{fleming-souganidis-86}
Wendell~H. Fleming and Panagiotis~E. Souganidis.
\newblock P{DE}-viscosity solution approach to some problems of large
  deviations.
\newblock {\em Ann. Scuola Norm. Sup. Pisa Cl. Sci. (4)}, 13(2):171--192, 1986.

\bibitem{freidlin-87}
M.~I. Freidlin.
\newblock Tunneling soliton in the equations of reaction-diffusion type.
\newblock In {\em Reports from the {M}oscow refusnik seminar}, volume 491 of
  {\em Ann. New York Acad. Sci.}, pages 149--156. New York Acad. Sci., New
  York, 1987.

\bibitem{freidlin-85}
Mark Freidlin.
\newblock Limit theorems for large deviations and reaction-diffusion equations.
\newblock {\em Ann. Probab.}, 13(3):639--675, 1985.

\bibitem{freidlin-92}
Mark~I. Freidlin.
\newblock Semi-linear {PDE}s and limit theorems for large deviations.
\newblock In {\em \'{E}cole d'\'{E}t\'e de {P}robabilit\'es de {S}aint-{F}lour
  {XX}---1990}, volume 1527 of {\em Lecture Notes in Math.}, pages 1--109.
  Springer, Berlin, 1992.

\bibitem{jabin-raoul-11}
Pierre-Emmanuel Jabin and Ga{\"e}l Raoul.
\newblock On selection dynamics for competitive interactions.
\newblock {\em J. Math. Biol.}, 63(3):493--517, 2011.

\bibitem{perthame-lions}
P.-L. Lions and B.~Perthame.
\newblock Remarks on {H}amilton-{J}acobi equations with measurable
  time-dependent {H}amiltonians.
\newblock {\em Nonlinear Anal.}, 11(5):613--621, 1987.

\bibitem{lions-82}
Pierre-Louis Lions.
\newblock {\em Generalized solutions of {H}amilton-{J}acobi equations},
  volume~69 of {\em Research Notes in Mathematics}.
\newblock Pitman (Advanced Publishing Program), Boston, Mass.-London, 1982.

\bibitem{lorz-mirrahimi-al-11}
Alexander Lorz, Sepideh Mirrahimi, and Beno{\^{\i}}t Perthame.
\newblock Dirac mass dynamics in multidimensional nonlocal parabolic equations.
\newblock {\em Comm. Partial Differential Equations}, 36(6):1071--1098, 2011.

\bibitem{lunardi}
Alessandra Lunardi.
\newblock {\em Analytic semigroups and optimal regularity in parabolic
  problems}, volume~16 of {\em Progress in Nonlinear Differential Equations and
  their Applications}.
\newblock Birkh\"auser Verlag, Basel, 1995.

\bibitem{metz-geritz-al-96}
J.~A.~J. Metz, S.~A.~H. Geritz, G.~Mesz{\'e}na, F.~J.~A. Jacobs, and J.~S. van
  Heerwaarden.
\newblock Adaptive dynamics, a geometrical study of the consequences of nearly
  faithful reproduction.
\newblock In {\em Stochastic and spatial structures of dynamical systems
  (Amsterdam, 1995)}, Konink. Nederl. Akad. Wetensch. Verh. Afd. Natuurk.
  Eerste Reeks, 45, pages 183--231. North-Holland, Amsterdam, 1996.

\bibitem{mirrahimi-perthame-al-12}
Sepideh Mirrahimi, Beno{\^{\i}}t Perthame, and Joe~Yuichiro Wakano.
\newblock Evolution of species trait through resource competition.
\newblock {\em J. Math. Biol.}, 64(7):1189--1223, 2012.

\bibitem{CRAS-mirrahimi-roquejoffre}
Sepideh Mirrahimi and Jean-Michel Roquejoffre.
\newblock Uniqueness in a class of {H}amilton-{J}acobi equations with
  constraints.
\newblock {\em C. R. Math. Acad. Sci. Paris}, 353(6):489--494, 2015.

\bibitem{mirrahimi-roquejoffre-16}
Sepideh Mirrahimi and Jean-Michel Roquejoffre.
\newblock A class of {H}amilton-{J}acobi equations with constraint: uniqueness
  and constructive approach.
\newblock {\em J. Differential Equations}, 260(5):4717--4738, 2016.

\bibitem{pazy}
Amnon Pazy.
\newblock {\em Semi-groups of linear operators and applications to partial
  differential equations}.
\newblock Department of Mathematics, University of Maryland, College Park, Md.,
  1974.
\newblock Department of Mathematics, University of Maryland, Lecture Note, No.
  10.

\bibitem{perthame15}
Beno{\^{\i}}t Perthame.
\newblock {\em Parabolic equations in biology}.
\newblock Lecture Notes on Mathematical Modelling in the Life Sciences.
  Springer, Cham, 2015.
\newblock Growth, reaction, movement and diffusion.

\bibitem{perthame-barles-08}
Beno{\^{\i}}t Perthame and Guy Barles.
\newblock Dirac concentrations in {L}otka-{V}olterra parabolic {PDE}s.
\newblock {\em Indiana Univ. Math. J.}, 57(7):3275--3301, 2008.

\bibitem{perthame-genieys-07}
Beno{\^{\i}}t Perthame and Stephane G{\'e}nieys.
\newblock Concentration in the nonlocal {F}isher equation: the
  {H}amilton-{J}acobi limit.
\newblock {\em Math. Model. Nat. Phenom.}, 2(4):135--151, 2007.

\bibitem{raoul-11}
Ga{\"e}l Raoul.
\newblock Long time evolution of populations under selection and vanishing
  mutations.
\newblock {\em Acta Appl. Math.}, 114(1-2):1--14, 2011.

\bibitem{vinter-wolenski}
R.~B. Vinter and P.~Wolenski.
\newblock Hamilton-{J}acobi theory for optimal control problems with data
  measurable in time.
\newblock {\em SIAM J. Control Optim.}, 28(6):1404--1419, 1990.

\end{thebibliography}
\end{document}